\documentclass[a4paper, 9pt]{article}
\usepackage{graphics}
\usepackage[dvips]{color}

\usepackage{latexsym}
\usepackage{amsmath,verbatim}
\usepackage{graphics}
\usepackage{amsthm}
\usepackage{amssymb}
\usepackage{makeidx}
\usepackage{stmaryrd}
\usepackage{multicol}
\usepackage[all]{xy}
\newfont{\Fr}{eufm10}
\newfont{\Sc}{eusm10}
\newfont{\Bb}{msbm10}
\newfont{\Am}{msam10}
\newfont{\am}{msam7}
\numberwithin{equation}{section}
\newtheorem{theorem}{Theorem}[section]
\newtheorem{proposition}[theorem]{Proposition}
\newtheorem{lemma}[theorem]{Lemma}
\newtheorem{corollary}[theorem]{Corollary}
\newtheorem{claim}{Claim}{\bf}{\it}

\newtheorem{ftheorem}{Theorem}{\bf}{\it}
{\bf}{\it}
\theoremstyle{definition}
\newtheorem{definition}[theorem]{Definition}
\newtheorem{algorithm}[theorem]{Algorithm}

\newtheorem{convention}[theorem]{Convention}

{\bf}{\rm}

\theoremstyle{remark}
\newtheorem{example}[theorem]{Example}
\newtheorem{remark}[theorem]{Remark}

\newtheorem{definition and corollary}[theorem]{Definition and Corollary}

{\it}{\rm}

\newcommand{\Hom}{\mbox{\rm Hom}}

\newcommand{\Ext}{\mbox{\rm Ext}}

\newcommand{\End}{\mbox{\rm End}}

\newcommand{\Ind}{\mbox{\rm Ind}}

\newcommand{\h}{\mathfrac{\h}}

\newcommand{\fd}{\mbox{\rm fd}}
\newcommand{\gd}{\mbox{\rm gd}}
\newcommand{\sgn}{\mathsf{sgn}}
\newcommand{\triv}{\mbox{\rm 1}}
\newcommand{\ds}{\mathsf{ds}}
\newcommand{\EP}{\mbox{\rm EP}}

\title{On characters and formal degrees of discrete series of affine Hecke algebras of classical types}
\author{Dan \textsc{Ciubotaru}\footnote{Department of Mathematics, University of Utah, Salt Lake City, UT 84112, USA}, Midori \textsc{Kato} {\small\textsc{(Shiota)}}\footnote{Deceased. 1981.06.28--2010.02.17.} \footnote{Graduate School of Mathematical Sciences, University of Tokyo, 3-8-1 Komaba, Meguro 154-8914 Japan}, \text{ and } Syu \textsc{Kato}\footnote{Research Institute for Mathematical Sciences, Kyoto
  University, Oiwake Kita-Shirakawa Sakyo Kyoto 606-8502, Japan.} \footnote{{\it Current Address:} Department of Mathematics, Kyoto
  University, Oiwake Kita-Shirakawa Sakyo Kyoto 606-8502, Japan.}}

\begin{document}

\maketitle

\begin{abstract}
We address two fundamental questions in the representation theory of
affine Hecke algebras of classical types. One is an inductive algorithm to compute characters of tempered modules, and the other is the determination of
the constants in the formal degrees of discrete series (in the form
 conjectured by Reeder \cite{Re}). The former is
completely different from the Lusztig-Shoji algorithm
\cite{Sh, L}, and it is more effective in a number of cases. The main idea
in our proof is to introduce a new family of representations which
behave like tempered modules, but for which it is easier to analyze the effect of
parameter specializations. Our proof also requires a comparison of the
$C^{\ast}$-theoretic results of Opdam, Delorme, Slooten, Solleveld
\cite{O, DO, Sl2, OSa, OS}, and the geometric construction from
\cite{K1,K2,CK}. 
\end{abstract}



\section{Introduction}
In this paper, we consider two basic questions in the study of affine
Hecke algebra of classical types with unequal parameters. The first
one is the characters of tempered modules. The classical approach (for
$W$-characters) is via the Lusztig-Shoji algorithm (\cite{Sh,L}),
which computes the generalized Green functions in terms of geometric
data. We present an alternative approach, namely an inductive algorithm
on the rank and the ratio of parameters of the affine Hecke
algebra. Since the Lusztig-Shoji algorithm treats each (geometric)
ratio of parameters individually, our algorithm has some advantage,
particularly if one is interested in the connection between two
different ratios. As a consequence of this new algorithm, we answer
the second basic question, namely the determination of the rational
constants in the formal degrees of discrete series. Our result
confirms the expected values of these constants, motivated by the
study of complex smooth representations of $p$-adic groups (see the
discussion after Theorem \ref{fformal}). More generally, in
conjunction with Bushnell-Henniart-Kutzko \cite{BHK} Theorem B, this
provides an explicit formula for formal degrees of discrete series of
$p$-adic groups of classical types for many (if not all) Bernstein
blocks.   

To explain our results more precisely, let $\mathbb H_n(q,u,v)$ be the
affine Hecke algebra of type $\mathsf C_n$ with parameters $q,u,v$ (see \S
\ref{Hecke}). We specialize to the cases $\mathbb H_{n,m}=\mathbb
H_n(q,q^m,q^m)$ and $\mathbb H'_{n,m}=\mathbb H_n(q,q^{2m},1)$, with
$m\in \mathbb R$. These are the affine Hecke algebras
with two parameters of type $\mathsf C_n$ and (up to
central extension) of type $\mathsf B_n$,
respectively. Let $W_n$ denote the Weyl group of type
  $\mathsf{BC}_n,$ and let $\widehat W_n$ denote the set of
  irreducible $W_n$-representations. In order to explain our results on the $W_n$-character of tempered modules, we also restrict to the so-called positive real central character case (see section
\S \ref{rts}). 

There is a correspondence between the set of discrete series with real central characters of $\mathbb
H_{n,m}$ and $\mathbb H'_{n,m}$. For every partition $\sigma$ of $n$, there is a real central
character $\mathsf c_m^\sigma$ attached to $\sigma$ and $m$. When $m$
is generic, i.e., $m\notin \frac 12\mathbb Z$, there exists a unique
discrete series with central character $\mathsf c_m^\sigma$, and moreover, every
discrete series module of $\mathbb H_{n,m}$ or $\mathbb H'_{n,m} $ has
central character $\mathsf c_m^\sigma$ for some partition $\sigma$ of
$n$ (Opdam \cite{O}). Therefore, we can regard the discrete series with real central
character as belonging to families $\{\ds_m(\sigma)\}_m$,
$\{\ds'_m(\sigma)\}_m$ for $\mathbb H_{n,m}$ and $\mathbb H'_{n,m}$,
respectively, indexed by partitions $\sigma$ of $n$. Then, we have
\begin{equation}  
\ds_m(\sigma)\cong \ds'_m(\sigma),\text{ as } W_n\text{-modules}.
\end{equation}
If $m_0$ is a critical parameter, i.e., $m_0\in \frac 12\mathbb Z$,
then it is known by Opdam-Solleveld \cite{OS} that every discrete series of $\mathbb H_{n,m_0}$
(resp. $\mathbb H'_{n,m_0}$) is obtained as a limit $m\to m_0$ of
certain $\ds_m(\sigma)$ (resp. $\ds'_m(\sigma)$). Here,
  $\lim_{m\to m_0}$ is in the sense of \cite{CK} \S2.4.

As already mentioned, we are interested in the $W_n$-character of
$\ds_m(\sigma)$. Our strategy is as follows: for $m>n-1$,
the $\ds_m(\sigma)$ is simple and does not change as a $W_n$-module by \cite{CK}. Namely, we have
\begin{equation}
\ds_m(\sigma)\mid_{W_n}=\{\emptyset,{ }^{\mathtt t}\!\sigma\}, \text{ if }m>n-1,\label{fasym}
\end{equation}
where the notation for $\widehat W_n$ is via bipartitions (c.f. \S \ref{sec:Walg}). Then we keep track how the $W_n$-character of $\ds_m(\sigma)$ changes
as $m$ varies towards $-\infty$. It can only change when $m$
passes through a critical value $m_0$, but this is a subtle
problem. We resolve this difficulty by considering a
larger family of irreducible modules $\mathcal D_{m_0}(\sigma)$ depending on $m \neq m_0$, $m_0-\frac 12<m<m_0+\frac 12$ with the following properties:

\begin{enumerate}
\item[(a)] $\pi$ has central character $\mathsf c_m^\sigma$
  (same as $\ds_m(\sigma)$); and
\item[(b)] $\lim_{m\to m_0} \pi$ is tempered.
\end{enumerate}

For lack of a better name, we call such modules delimits of tempered modules, or tempered delimits for short. For example, we have $\ds_m(\sigma)\in \mathcal
D_{m_0}(\sigma)$ for both $m_0-\frac{1}{2}<m<m_0$ and $m_0<m<m_0+\frac{1}{2}$. It should be noted that the modules appearing as $\lim_{m\to m_0}\pi$, $\pi\in \mathcal D_{m_0}(\sigma)$ can be thought of as analogues of the nondegenerate limits of discrete series in the sense of Knapp-Stein \cite{Kn}, \S XIV.17 and Theorem 14.92.

\smallskip

A main technical achievement of this paper is the following:

\begin{ftheorem}[{Corollary} \ref{irred}]\label{firr}
Assume that $m_0\in\frac 12\mathbb Z$. Then for every $\pi\in \mathcal D_{m_0}(\sigma)$, $\lim_{m\to
  m_0}\pi$ is an irreducible $\mathbb H_{n,m_0}$-module. In particular, $\lim_{m\to m_0}\ds_m(\sigma)$ is irreducible.
\end{ftheorem}

This theorem, proved as a corollary of basic properties of tempered delimits (Theorems \ref{multiplicity formula}, \ref{minds}), represents the basis for our algorithm. By the geometry of tempered delimits, we deduce:

\begin{ftheorem}[Formula (\ref{transf})]\label{find}
For every $m_0\in\frac 12\mathbb Z$, we have the following equality inside the Grothendieck group of $\mathbb H _{n,m_0}$-modules:
\begin{eqnarray}
[ \lim _{m' \to m_0} \mathsf{ds} _{m'} ( \sigma )] \pm [ \lim _{m \to m_0} \mathsf{ds} _m ( \sigma ) ] = \sum (\pm) [L ^{\mathsf A} \boxast L'],\label{wc}
\end{eqnarray}
where the real variables $m,m'$ satisfies $m_0 - \frac{1}{2} < m' < m_0 < m < m_0 + \frac{1}{2}$. Here $L ^{\mathsf A} \boxast L'$ denotes parabolic induction from a tempered module $L ^{\mathsf A}$ of an affine Hecke algebra of type $\mathsf A$ and a discrete series $L'$ of a type $\mathsf C$ affine Hecke algebra. Moreover, all the terms in the right hand side are induced from proper Levi subalgebras.
\end{ftheorem}

We remark that the right hand side of (\ref{wc}) looks obscure here but the actual expression is explicit and precise (see (\ref{transf}) for details). Moreover, (\ref{wc}) implies certain relations between the $W$-characters of classical and exotic Springer fibers (Corollary \ref{diff}).

In addition, if we assume, by induction on the rank of the Hecke
algebra, that we know the discrete series character of smaller
affine Hecke algebras of type $\mathsf C$, then we easily deduce the
character of the right hand side of (\ref{wc}). Hence, if we
know the character of either $\mathsf{ds} _{m'} ( \sigma )$ or
$\mathsf{ds} _{m} ( \sigma )$, then we deduce the other. Thanks to
(\ref{fasym}), we always know the $W_n$-character of $\mathsf{ds} _m (
\sigma )$ for $m \gg 0$. Our algorithm \ref{algW} (on the $W_n$-characters of tempered delimits) is an implementation of these observations. As we see in Remark \ref{RTch}, our construction also gives an inductive algorithm to compute weight characters of tempered delimits with respect to the abelian subalgebra that appears in the Bernstein-Lusztig presentation (\cite{L2} \S3).

In section \ref{sec:formal}, we use the $W$-character algorithm to complete the
computation of the formal degree for the affine Hecke algebra $\mathbb
H_n(q,q^{m_+},q^{m_-})$ of type $\mathsf C_n$, where $q>1$ and $m_\pm\in
\mathbb R$. All affine Hecke algebras of classical types are (up to
central extensions) particular cases of this one. Denote the roots of
type $\mathsf C_n$ by $R_n$, and let $R_n^{\mathsf{sh}}$ and
$R_n^{\mathsf{lo}}$ denote the short and long roots, respectively. From \cite{OS}, the
formal degree of a discrete series $\pi$ with central character
$s$ (not necessarily positive real) is known to equal
\begin{equation}\label{fdprod}
\fd(\pi)=\frac{{C_\pi~ q^{n^2-n} q^{n
      m_+}\prod_{\alpha\in R_n}' (\alpha(s)-1)}}{{\prod_{\alpha\in
      R_n^{\mathsf {sh}}}' (q\alpha(s)-1) \prod_{\alpha\in
      R_n^{\mathsf {lo}}}' (q^{\frac
      {m_++m_-}2}\alpha(s)^{1/2}-1)\prod_{\alpha\in R_n^{\mathsf
        {lo}}}' (q^{\frac {m_+-m_-}2}\alpha(s)^{1/2}+1) }  },
\end{equation}
where $\prod'$ means that the product is taken only over the nonzero
factors. 
From Opdam-Solleveld \cite{OSa}, it is known that the constants $C_{\pi}$
are rational numbers, and the question is to
determine them explicitly. 
We use an Euler-Poincar\'e formula which expresses
the formal degree as an alternating sum depending on the $W$-character of the
discrete series (see \ref{e:EP}) as in Reeder \cite{Re}. This formula itself is proved in Schneider-Stuhler \cite{SS} (for $p$-adic groups) and in Opdam-Solleveld \cite{OSa} (for affine Hecke algebras).

Following \cite{K1}, we say that $(m_+,m_-)$ are generic if $|m_+\pm m_-|\notin \{0,1,2,\dots, 2n-1\}$. We use Theorem \ref{find} to find that the constants
$C_\pi$ for generic $(m_+,m_-)$ do not
depend (up to sign) on $m$. Combined with an explicit calculation in
an asymptotic region of the parameters $(m_+,m_-)$ and a certain limiting process, we obtain:

\begin{ftheorem}\footnote{This is a corrected version ($C_\pi=1$ not
    $1/2$) of the theorem
    that appeared in the published version. We thank Eric Opdam for
    pointing out the error to us.}
[Theorem \ref{mainformal} and Corollary
  \ref{c:mainformal}]\label{fformal} Let $\pi$ be a discrete series with arbitrary central character for the affine Hecke algebra $\mathbb H_n(q,q^{m_+},q^{m_-})$, where $q>1$ and $m_\pm\in \mathbb R.$ Then, the constant in (\ref{fdprod}) is
  (up to sign) $C_\pi=1$.
\end{ftheorem}
In \S \ref{s:mainformal} (\ref{typeC},\ref{typeB},\ref{typeD}), we explain the implications of Theorem \ref{fformal} for the affine Hecke algebras of types $\mathsf C_n,\mathsf
B_n,\mathsf D_n$, respectively.

\smallskip

As mentioned previously, this calculation has consequences for $p$-adic groups as well. The expected stability of $L$-packets of discrete series for a $p$-adic
group $\mathcal G$ implies that the formal degrees of discrete series in the same
$L$-packet have to be proportional, with the proportionality constants
being the multiplicities of discrete series in the stable $L$-packet
sum. In the case discrete series are in the scope of the Deligne-Langlands-Lusztig correspondence \cite{L4}, there is a precise conjecture for the values of the constants formulated in \cite{Re} (0.5). Particularly, when the $p$-adic group is of classical type (other than the quasisplit triality form of $D_4$), those discrete series are controlled by that of various affine Hecke algebras of classical types. For example, with our notation, the Iwahori cases for split $p$-adic classical groups $SO(2n+1)$, $Sp(2n),$ $SO(2n)$ 
correspond to the Hecke algebras $\mathbb H'_{n,\frac 12}$, $\mathbb
H_{n,1}$ and $\mathbb H'_{n,0}$, respectively. (In fact the last
algebra is central extension of the Iwahori-Hecke algebra for type
$\mathsf D_n$, but for our purposes, this is sufficient{; see Proposition
\ref{irrD}}). Using the correspondence between the Plancherel formula for
groups and for the Hecke algebras (\cite{BHK}), and taking also into account Hiraga-Ichino-Ikeda \cite{HII} \S 3.4, one verifies that the values of the constants obtained from Theorem \ref{fformal} match the expected values from $p$-adic groups.

\smallskip

The organization of the paper is as follows. In \S 2 we recall the geometric setup, and we fix the notation for the affine
Hecke algebras. Then we study a number of properties of Langlands
quotients of parabolically induced modules which we need in \S
3. In \S 3, we define and classify the tempered delimits, and prove
the results about irreducibility under deformations in the parameter
$m$. We present the inductive algorithm for the $W$-characters of
discrete series and tempered modules. In \S 4, we calculate the
constants in formal degrees.

\smallskip

\noindent{\bf Acknowledgments.} We would like to thank E. Opdam for
suggesting the problem of calculating the constants in the formal
degrees and for sharing with us his insight into this
  problem. We are also grateful to S. Ariki and T. Shoji for helpful
discussions about this project. The work on this paper began during
the conference and workshop ``Representation Theory of Real Reductive
Groups'' at University of Utah; we thank the organizers for the
invitation and support. D.C. was partially supported by NSF-DMS 0554278 and NSA-AMS 081022 and S.K. by the Grant-in-Aid for Young Scientists (B) 20-74011.

\

\begin{flushleft}
{\bf Convention}
\end{flushleft}
For two sets $J _1, J _2 \subset \mathbb Z$, we define $J _1 < J _2$ if and only if $j_1 < j _2$ for every $j_1 \in J_1$ and $j_2 \in J_2$.

Fix $\vec{q} = ( q _1, q ) \in \mathbb R ^2$ so that $q > 1$ and $q_1 = q ^{m}$ for some $m \in \mathbb R $. We say $m$ is {\it generic} if and only if $m \not\in \frac{1}{2} \mathbb Z$. A $q$-segment (or just a segment if there can be no possible confusion) is a sequence of positive real numbers of the form
$$a, a q, a q^2, \cdots, a q ^M\text{ for some }M \in \mathbb Z _{\ge 0}.$$
For two $q$-segments $I _1, I _2$, we define
\begin{align*}
& \mathtt E _m ( I_1 ) := & \prod _{a \in I _1} a, & \hskip 5mm e ^+ _{m} ( I _1 ) = e _+ ( I _1 ) :=  \max I _1, \hskip 5mm e ^- _{m} ( I _1 )= e _- ( I _1 ) := \min I_1,\\
& I _1 \Subset I _2 & \Leftrightarrow & \min I _2 < \min I _1 \text{ and } \max I _1 < \max I _2,\\
& I _1 \triangleleft I _2 & \Leftrightarrow & \min I _1 < \min I _2 \le q \max I _1 < q \max I _2, \text{ and }\\
& I _1 \trianglelefteq I _2 & \Leftrightarrow & \min I _1 \le \min I _2 \le q \max I _1 < q \max I _2\\
& & \text{ or } & \min I _1 < \min I _2 \le q \max I _1 \le q \max I _2.
\end{align*}

Finite collections of $q$-segments (with possible repetitions) are called $q$-multisegments (or just multisegments). The set of $q$-multisegments is denoted by $\mathsf{Q} ( q )$. For $\mathbf I, \mathbf I' \in \mathsf{Q} ( q )$, we write $\mathbf I \subset \mathbf I'$ if each segment of $\mathbf I$ gives a segment of $\mathbf I'$ with multiplicity counted.

For a partition $\lambda$, we set $| \lambda | := \sum _i \lambda _i$, $\lambda _i ^{<} := \sum _{j < i} \lambda _j$, and $\lambda _i ^{\le} := \sum _{j \le i} \lambda _j$. We denote by ${} ^{\mathtt t} \lambda$ the transpose partition of $\lambda$.

For an algebraic variety $\mathcal X$ over $\mathbb C$, we denote by $H _{\bullet} ( \mathcal X )$ the total Borel-Moore homology with coefficients in  $\mathbb C$.

\section{Preliminaries}

\subsection{Basic geometric setup}
We denote by $G _n = \mathop{Sp} ( 2n, \mathbb C )$ the symplectic
group with its maximal torus $T _n$ and a Borel subgroup $B _n \supset
T_n$. Let $R_n \supset R ^+_n$ be the root systems of $(G_n,T_n)$ and
  $(B_n,T_n)$, respectively. We define $X ^* ( T_n )$ to be the character lattice of $T _n$ with its natural orthonormal basis $\epsilon _1, \ldots, \epsilon _n$ so that
\begin{align*}
& R ^+_n = \{(\epsilon_i \pm \epsilon_j), i < j, 2\epsilon_i\} \subset R _n =\{\pm(\epsilon_i\pm\epsilon_j), i < j, \pm 2\epsilon_i\}\\
& \check{R} ^+ _n = \{(\epsilon_i \pm \epsilon_j), i < j, \epsilon_i\} \subset \check{R} _n =\{\pm(\epsilon_i\pm\epsilon_j), i < j, \pm \epsilon_i\},
\end{align*}
where $\check{R} _n \supset \check{R} ^+ _n$ is the dual root system. Let $\check\alpha \in \check{R} _n$ denote the coroot of $\alpha \in R _n$. Let $W _n := N _{G_n} ( T_n ) / T_n$ be the Weyl group of $G_n$. Let $V _n ^{(1)} = \mathbb C ^{2n}$ be the vector representation of $G _n$ and let $V_n ^{(2)} := \wedge ^2 \mathbb C ^{2n}$ be its second wedge. We define $\mathbb V _n := V _n
  ^{(1)} \oplus V _n ^{(2)}$ to be the $1$-exotic representation of $G
  _n$. Let $\mathbb V _n ^+$ be the sum of $T _n$-weight spaces of
  $\mathbb V _n$ for which the corresponding weights are in $\check{R} ^+ _n$. We 
  define $W _n ( s ) := \{ w \in W _n \mid \mathrm{Ad} ( w ) s = s \}$
  for each $s \in T _n$. For $w\in W _n$, we fix a lift $\dot{w}$ of $w$ in $N _{G} ( T )$. We set $F _n := G _n \times ^{B _n} \mathbb V _n ^+$. We form a map 
\begin{equation}
\mu _n : F_n = G _n \times ^{B _n} \mathbb V _n ^+ \longrightarrow \mathbb V _n \label{emoment}
\end{equation}
obtained as the anti-diagonal free $B _n$-quotient of the action map $G _n \times \mathbb V _n ^+ \to \mathbb V _n$. For every semisimple element $a=(s,\vec q)$, denote by $F _n^a$, $\mathbb V_n^a$, and $\mu _n ^a$, the $a$-fixed
  points and the restriction to the $a$-fixed points of $F _n, \mathbb V_n,$ and $\mu _n$, respectively. Moreover, for every subvariety $Y$ we denote by $Y^a$, the intersection of $Y$ with the $a$-fixed points.

We might drop the subscript $n$ when the meaning is clear from the context.

\subsection{Hecke algebras}\label{Hecke}
We consider the affine Hecke algebras $\mathbb H_n'(q,u)$, $\mathbb H_n(q,u)$, and $\mathbb H_n''(q)$ of type
${\mathsf{B}}_n$, ${\mathsf{C}}_n$, and ${\mathsf{D}}_n$, respectively, with positive real parameters $u,v$, according to the affine
Coxeter diagrams 
\begin{align*}
&\widetilde{\mathsf{B}}_n:\quad \xymatrix{q\ar@{-}[r] & q\ar@{-}[r]& q\ar@{-}[r]& \dots \ar@{-}[r]&q\ar@{-}[r]&q\ar@{=}[r]&u\\
& q\ar@{-}[u],}\\
&\widetilde{\mathsf{C}}_n:\quad \xymatrix{u\ar@{=}[r]&q\ar@{-}[r] &q\ar@{-}[r]&\dots\ar@{-}[r]
&q\ar@{-}[r] &q\ar@{=}[r] &u,} \text{ and }\\
&\widetilde{\mathsf{D}}_n:\quad \xymatrix{q\ar@{-}[r] & q\ar@{-}[r]& q\ar@{-}[r]& \dots \ar@{-}[r]&q\ar@{-}[r]&q\ar@{-}[r]&q\\
& q\ar@{-}[u] & & & & q\ar@{-}[u] &}.
\end{align*}
We consider them as subalgebras of certain specializations (see below) of the affine Hecke algebra $\mathbb H_n(q,u,v)$ of type ${\mathsf{C}_n}$
\begin{equation*}
\widetilde{\Pi}_n:\quad \xymatrix{v\ar@{=}[r]&q\ar@{-}[r] &q\ar@{-}[r]&\dots\ar@{-}[r]
&q\ar@{-}[r] &q\ar@{=}[r] &u}
\end{equation*}
defined as a $\mathbb C$-algebra with the set of generators $N_0, N_1, \ldots, N_n$ subject to the relations:
\begin{itemize}
\item $ ( N _0 + 1 )( N _0 - v ) = 0 = ( N _n + 1 )( N _n - u )$ and $( N _i + 1 )( N _i - q ) = 0$ for $1 \le i < n$; 
\item $N _i N _j = N _j N _i$ for $i - j \ge 2$, $N _i N _{i+1} N _i = N _{i+1} N _i N _{i+1}$ for $1 \le i < n - 1$;
\item $( N_0 N_1 )^2 = ( N_1 N_0 )^2$ and $( N_{n-1} N_n )^2 = ( N_n N_{n-1} )^2$.
\end{itemize}

Let $\mathbb H _n ^{\mathsf A}$ be the affine Hecke algebras of type $\mathop{GL} ( n )$ with
parameter $q$, which can be realized as a subalgebra of $\mathbb H _n (q,u,v)$ generated by $N_1, \ldots, N _{n-1}$, and $N_1 N _2 \cdots N_{n-1} N_n N _{n-1} \cdots N _1 N _0 ^{-1}$.

We remark that $\mathbb H _{n} ( q,u )$ is obtained from $\mathbb H_n(q,u,v)$ by making the specialization $u=v$. 
Also, a central extension $\mathbb H_n ^{\mathsf B} (q,u)$ of $\mathbb H_n'(q,u)$ is obtained from $\mathbb H_n(q,u,v)$ by making the specialization $v=1$.

We define an algebra involution $\psi : \mathbb H_n(q,1,1) \to \mathbb H_n(q,1,1)$ as:
$$\psi ( N _i ) = N _i \text{ if } i \neq n, \text{ and } \psi ( N _n ) = - N _n.$$
A central extension $\mathbb H_n ^{\mathsf D} = \mathbb H_n ^{\mathsf D} (q)$ of $\mathbb H_n''(q)$ is realized as the $\psi$-invariant part of $\mathbb H_n(q,1,1)$ (see for example \cite{OS} \S 6.4).

We denote the finite Weyl groups of type $\mathsf{BC}_n$ and $\mathsf{D} _n$ by
$W_n$ and $W _n ^{\mathsf D}$, respectively. We denote the affine Weyl groups of type $\mathsf{B}_n$ and $\mathsf{C}_n$ by $\widetilde{W} _n$ and $\widetilde{W}'_n$, respectively.

We define $\mathbb H _{n,m} := \mathbb H_n(q,q^m)$, and $\mathbb
H_{n,m} ^{\mathsf B} := \mathbb H_n'(q,q^{2m})$. The representation
theories of $\mathbb H _{n,m}$ and $\mathbb H_{n,m} ^{\mathsf B}$ are
known to be equivalent {to that of $\mathbb H_n (q,u,v)$ with $u = q^{(m + m')}$ and $v = q^{(m - m')}$ for an arbitrary $m' \in \mathbb R$}, once we fix a positive real central
character (Lusztig \cite{L2,L5}, see also \cite{K1} \S3 and \S \ref{rts} below for the geometric explanation). Moreover, these
equivalences preserve $W_n$-characters, and the notion of tempered modules and discrete series (see for example \cite{L5}, \S3). Since a central extension does not have an effect at the level of representations with positive real central character, we only deal with the representation theory of $\mathbb H_{n,m}$ and
$\mathbb H_{n} ^{\mathsf D}$ in this section and \S \ref{delimits}. In addition, we sometimes drop the subscript $m$ for the sake of simplicity.

We also need in \S \ref{sec:formal} the finite Hecke algebra of type $\mathsf{BC}_n$ with
parameters $q,u$ according to the Coxeter diagram
\begin{equation*}
\xymatrix{q\ar@{-}[r] &q\ar@{-}[r]&\dots\ar@{-}[r]
&q\ar@{-}[r] &q\ar@{=}[r] &u}
\end{equation*}
we denote it by $\mathbb H_n^f(q,u),$ or by $\mathbb
H_{n,m}^f$ when $u=q^{m}$. We denote by $\mathbb H_{n}^{\mathsf D, f}$
the $\psi$-invariant part of $\mathbb H_n^f(q,1)$. We remark that the
irreducible modules of $\mathbb H_{n,m}^f$ and $\mathbb H_{n}^{\mathsf D, f}$ are in one-to-one correspondence with $\widehat {W}_n$ and $\widehat{W_n^{\mathsf D}}$, respectively.

Let $R$ be a ring. Let $M$ be a $R$-module and let $L$ be an irreducible $R$-module. Then, we denote the Jordan-H\"older multiplicity of $L$ in $M$ as $R$-modules by $[M:L]_R$. If $R = \mathbb H _{n,m}$, then we drop the subscript $R$ for the sake of simplicity.

\subsection{Representation-theoretic setup}\label{rts}
A result of Bernstein and Lusztig says that the center of $\mathbb H_n$ is
\begin{equation}
Z(\mathbb H_{n}) =\mathbb C[e^\lambda;\lambda\in X^*(T_n)]^{W_n},
\end{equation}
so the central characters of $\mathbb H_{n}$ are parameterized by
$W_n$-conjugacy classes of semisimple elements $s \in T_n$. An element (or a central character) $s \in T _n$ is said to be positive real if {$\epsilon _i ( s ) > 0$} for $i=1,\ldots,n$.

We denote by $\mathfrak{Mod} _{\vec{q}} ^n$ the category of
finite-dimensional $\mathbb H _n$-modules with positive real central character. We set
$\mathfrak{Mod} _{\vec{q}} := \bigcup _{n \ge 1}
\mathfrak{Mod}_{\vec{q}}^n$. For a group $H$ and $h \in H$, we denote
by $R ( H )$ and $R ( H ) _h$ the representation ring of $H$ and its
localization along $h$, respectively.

For an $R ( T _n )$-module $M$, let $\Psi ( M ) \subset T _n$ denote the set of $R ( T _n )$-weights of $M$.  Moreover, we define $M [ s ] := R ( T_n ) _{s} \otimes _{R ( T_n )} M$ and
$$\mathsf{ch} M := \sum _{s\in T _n} \dim M [s] \left< s \right> \in \mathbb Z \left< T _n \right>,$$
where $\mathbb Z \left< T _n \right>$ is a formal linear combination
of elements of $T _n.$ 

We set
$$T _n ( \vec{q} ) := \{ s \in T _n ( \mathbb R ) \mid \epsilon _i ( s ) \in q _1 q ^{\mathbb Z} \text{ for each } i= 1, \ldots, n \}.$$
For every $s \in T _n ( \vec{q} )$, we define $v _s \in \mathfrak S _n$ as the minimal length element such that
$$v_s s := \mathrm{Ad} ( \dot{v _s} ) s \in T _n ( \vec{q} ) \text{ satisfies } \epsilon _1 ( v _s s ) \ge \epsilon _2 ( v _s s ) \ge \cdots \ge \epsilon _n ( v _s s ).$$
A marked partition ${\tau =} ( \mathbf J, \delta )$ of $n$ is a pair consisting of a collection $\mathbf J = \{ J _1, J _2, \ldots \}$ and a function $\delta : \{ 1, \ldots, n \} \to \{ 0,1 \}$ which satisfies
$$\bigsqcup _{j \ge 1} J _j = \{ 1, \ldots, n \}, \text{ and }\delta ( i ) = 1\text{ for at most one } i \in J \text{ for each } J \in \mathbf J.$$
For simplicity, we may denote $J _j \in \tau$ instead of $J _j \in \mathbf J$. For a marked partition $\tau =( \mathbf J, \delta )$, we define $\mathbf{v} _{\tau} = \mathbf{v} ^1 _{\tau} \oplus \mathbf{v} ^2 _{\tau}$ with
$$\mathbf{v} ^1 _{\tau} = \sum _{i \in  \{ 1, \ldots, n \}} \delta ( i ) \mathbf{v} _i \text{ and } \mathbf{v} ^2 _{\tau} = \sum _{J \in \mathbf J} \sum _{i, j \in J}  \delta _1 ( \# \{ k \in J \mid i \le k < j \} ) \mathbf{v} _{i,j},$$
where $\delta _1 ( i ) = 1$ ($i=1$) or $0$ ($i \neq 1$), and $\mathbf{v} _i \in V ^{(1)}, \mathbf{v} _{i,j} \in V ^{(2)}$ are $T$-eigenvectors of weights $\epsilon _i, \epsilon _i - \epsilon _j$, respectively. We put $\mathcal O _{\tau} := G \mathbf{v} _{\tau} \subset \mathbb V$.

We set $G ( \chi ) = G ( s ) := Z _G ( s )$. A marked partition $\tau$ is adapted to $a = (s, \vec{q} )$ or $s$ if we have $s \mathbf{v} _{\tau} = q _1 \mathbf{v} ^1 _{\tau} \oplus q \mathbf{v} ^2 _{\tau}$. We set $\mathsf{P} _n ( \vec{q} )$ as the set of pairs $\chi = (s,
\tau)$ consisting of $s \in T ( \vec{q} )$ and a marked partition
$\tau$ adapted to $s$. For $J \in \tau$, we put $\underline{J} := \{ \epsilon _j ( s ) \mid j \in J \}$, which we regard as a ($q$-)segment. We write $I \in \chi$ if $I = \underline{J}$ for some $J \in \tau$. We set $\mathcal O _{\chi} := \dot{v _s} G ( s ) \mathbf{v} _{\tau}$. Two marked partitions $\tau, \tau'$ adapted to $s$ are called equivalent (and we denote this by $\tau \sim \tau'$) if 
\begin{equation}\label{equiv}
\mathcal O _{(s,\tau)} = \mathcal O _{(s,\tau')}.
\end{equation}
This notion of equivalence can be translated in combinatorial terms on marked partitions; details are found in \cite{CK} \S1.4.

Two parameters $\chi,\chi'$ are called nested to each other if $I \Subset I'$  or $I' \Subset I$ holds for each $( I,I' ) \in \chi \times \chi'$.

For $\chi \in \mathsf{P} _n ( \vec{q} )$, let us denote by $\mathcal E _{\chi}$ the projection of $\mu
  _n ^{-1} ( \mathbf{v} _{\tau} ) ^{a}$ to $G _n / B _n$. Then, $M _{\chi} := H _{\bullet} ( \mathcal E _{\chi} )$ admits a structure of a module over the specialized algebra $\mathbb H _{a} = \mathbb H _{s} :=\mathbb H_n\otimes_{Z(\mathbb H_n)} \mathbb C_s$. We call $M_\chi$ the standard module attached to $\chi$ (cf. \cite{K1}).
We denote the irreducible $\mathbb H _{a}$-module corresponding to $\chi$ by $L _{\chi}$, which is a
quotient of $M _{\chi}$. We have a disjoint decomposition
$$\mathcal E _{\chi} = \bigsqcup _{s' \in W s \subset T _n} \mathcal E _{\chi} [ s' ], \text{ which induces } M _{\chi} = \bigoplus _{s' \in W s \subset T _n} M _{\chi} [ s' ] = \bigoplus _{s' \in W s \subset T _n} H _{\bullet} ( \mathcal E _{\chi} [s'] ).$$

Let $\mathfrak M _{\vec{q}}^n \subset \mathfrak{Mod} _{\vec{q}} ^n$ denote the full subcategory generated by simple modules corresponding to $\mathsf{P} _n ( \vec{q} )$. This is the category of $\mathbb H_{n,m}$-modules with central characters in $T ( \vec{q})$ (cf. \cite{K1}). We denote by $K (\mathfrak{M} _{\vec{q}} ^n)$ its Grothendieck group.

We put $\mathsf{P} (
\vec{q} ) := \cup _{n \ge 1} \mathsf{P} _n ( \vec{q} )$. We have a natural map $\mathsf{P} ( \vec{q} ) \mapsto \mathsf Q ( q )$ sending a pair $( s, \tau )$ with $\tau = ( \mathbf J, \delta )$ to $\{ \underline{J} \mid J \in \mathbf J \}$.
 We sometimes identify $\mathbf I \in \mathsf{Q} ( q )$ with its preimage in $\mathsf{P} ( \vec{q}
  )$ with trivial markings. We denote the set of such preimages by $\mathsf{P}^0 ( \vec{q} )$. We denote the size of a marked partition
by $| \tau |$ (or $| \chi |$). 

{Similarly, for $\mathbf I \in \mathsf{Q} ( q )$, we denote the corresponding standard and irreducible $\mathbb H _{n} ^{\mathsf A}$-modules by $M ^{\mathsf A} _{\mathbf I}$ and $L ^{\mathsf A} _{\mathbf I}$, respectively. For a segment $I$, we define its transpose to be the segment ${} ^{\mathtt t}I=\{ b ^{-1} \mid b \in I \}.$ For a multisegment $\mathbf I$, we define its transpose ${} ^{\mathtt t} \mathbf I$ to be the multisegment $\{{}^{\mathtt t}I\mid I\in\mathbf I \}$ (with multiplicity counted). We sometimes denote $M ^{\mathsf A} _{{} ^{\mathtt t}\mathbf I}$ or $L ^{\mathsf A} _{{} ^{\mathtt t}\mathbf I}$ by ${} ^{\mathtt t}M ^{\mathsf A} _{\mathbf I}$ and ${} ^{\mathtt t}L ^{\mathsf A} _{\mathbf I}$, respectively.

We define ${} ^0 \mathfrak{Mod} _{\vec{q}} ^n$ to be the category of
finite-dimensional $\mathbb H _n ^{\mathsf A}$-modules with positive real central characters. For each $\nu \in \mathbb R$, let $\mathsf{St} _n ^{\nu}$ denote the (central) twists of Steinberg representation
of $\mathbb H _n ^{\mathsf A}$ so that the corresponding unique segment is $I _n ^{\nu} := \{ q ^{\nu}, q ^{\nu + 1}, \ldots, q ^{\nu + n - 1} \}$. We denote the central character of $\mathsf{St} _n ^{\nu}$ by $\mathsf s ^{\nu} _n$ and the central character of ${} ^{\mathtt t} \mathsf{St} _n ^{\nu} = {} ^{\mathtt t}( \mathsf{St} _n ^{\nu} )$ by $\bar{\mathsf s} ^{\nu} _n$.

We have an exact functor
$${} ^0 \mathfrak{Mod} _{\vec{q}} \times \mathfrak{Mod} _{\vec{q}} \ni ( M _1, M _2 ) \mapsto M _1 \boxast M _2 \in \mathfrak{Mod} _{\vec{q}}$$
given by the parabolic induction. By abuse of notation, we also denote the parabolic induction of type $\mathsf{A}$ affine Hecke algebras as
$${} ^0 \mathfrak{Mod} _{\vec{q}} \times {} ^0 \mathfrak{Mod} _{\vec{q}} \ni ( M _1, M _2 ) \mapsto M _1 \boxast M _2 \in {} ^0 \mathfrak{Mod} _{\vec{q}}.$$

For $\chi _i = ( s_i, \mathbf J _i, \delta _i ) \in \mathsf{P} _{n_i} ( \vec{q} )$ ($i=1,2$), we set $\chi _1 \oplus \chi _2 := ( s_1 \times s _2, \mathbf J _1 \sqcup \mathbf J _2 [ n_1 ], \delta _{12} )$, where $\mathbf J _2 [n_1]$ is the collection of subsets of $\{n_1 + 1, \ldots, n_1 + n _2\}$ obtained from $\mathbf J _2$ by uniformly adding $n_1$, and $\delta _{12}$ is the marking such that $\delta_{12} \mid _{\mathbf J _1} = \delta _{1}$ and $\delta_{12} ( k ) \mid _{\mathbf J _2[n_1]} = \delta _{2} ( k - n_1 )$ for each $k$.

\subsection{Quotients of parabolic induction}
The goal of this subsection is Proposition \ref{iest}, which gives necessary conditions for an irreducible $\mathbb H_{n,m}$-module to appear as the quotient of a parabolically induced module. Before we prove this result, we need to fix notation and recall known results about quiver representations of type $\mathsf A$. Throughout this subsection, we assume that $m$ is {\it generic}.

For $\chi = ( s, \tau ) \in \mathsf{P} ( \vec{q} )$ such that $\tau = ( \mathbf J, \delta )$, we associate $\chi ^0 = ( s, \tau^0 ) \in \mathsf{P}^0 ( \vec{q} )$ with $\tau ^0 = ( \mathbf J, 0 )$.

Let $W [ \chi ]$ be the set of elements $w$ of $W _{| \chi |}$ such that $w
^{-1} s ^{-1} \in \Psi ( L _{\chi} )$.  For $\vec{n} = (n_1, n_2)$, $n=n_1+n_2$ with $n_1,n_2 \ge 0$, we define $\mathfrak S ^{\vec{n}}$ (resp. $W ^{\vec{n}}$) as a set of minimal length representative of $\mathfrak S _{n} / ( \mathfrak S _{n_1} \times \mathfrak S _{n_2} )$ (resp. $W _{n} / ( \mathfrak S _{n_1} \times W _{n_2} )$) inside $\mathfrak S _{n}$ (resp. $W _{n}$).

For $\chi' = ( s', \tau' ) \in \mathsf{P} (\vec{q})$, we say $\chi \le \chi'$ if and only if $v_s s = v _{s'} s'$
and $\mathcal O _{\chi} \subset \overline{\mathcal O _{\chi'}}$. We refer to this (partial) ordering as the closure ordering. We define
$$W [ \chi ] ^{\circ} := W [ \chi ] - \bigcup _{\chi' > \chi} W [ \chi' ].$$

For a pair $( \chi _1, \chi _2 ) \in \mathsf{Q} ( q ) \times \mathsf{P} ( \vec{q} )$, we define
\begin{equation}
W [ \chi_1, \chi_2 ] := \{ ( w_1 \times w _2 ) \in W _{| \chi _1 |
 + | \chi _2 |} \mid w _i \in W [ \chi _i ] ^{\circ} \}.
\end{equation}

\begin{lemma}
For each $\chi \in \mathsf{P} ( \vec{q} )$, we have $W [ \chi ] ^{\circ} \neq \emptyset$. Moreover, we have $\mathfrak S _n \cap W [ \chi ] ^{\circ} \neq \emptyset$ if $\chi = {} ^{\mathtt t} \mathsf{St} _n ^{\nu + m}$ for some $\nu \in \mathbb Z$.
\end{lemma}
\begin{proof}
For the first assertion, it is enough to choose $w \in W$ so that the conditions of \cite{CK} Proposition 4.9 are satisfied, and this is straight-forward. The second assertion is also straight-forward since $\chi$ corresponds to a regular nilpotent orbit in $\mathfrak{gl} _n$.
\end{proof}

For $s \in T _n ( \vec{q} )$, let $\mathbf E ^{s} ( i )$ denote the $s$-eigenspace of
$V ^{(1)} _n$ with its eigenvalue $q_1 q^i$. We have a natural identification
$$\mathbb V _{n} ^{(s,\vec{q})} \cong \mathbf E ^{s} ( 0 ) \oplus
\mathbf{Rep} ^{s},\text{ where }\mathbf{Rep}^{s}= \bigoplus _{i \in \mathbb Z} \mathrm{Hom} ( \mathbf E ^{s} ( i ), \mathbf E ^{s} ( i+1 ) ),$$
compatible with the $G _n ( s )$-action.

For each $w \in W _{n}$, we set  ${}^w\mathbb V_n^+=\dot w^{-1}\mathbb V_n^+$, and let us denote by $\mathbf{Rep} ^{s} _w$ the image of $( {} ^w \mathbb V ^+ _{n} \cap \mathbb V _{n} ^{(v_s s,\vec{q})} )$ in $\mathbf{Rep} ^{s}$ under the projection map $\mathbb V ^{(s,\vec{q})} \to \mathbf{Rep} ^{s}$. By abuse of notation, in place of $\mathbf E ^{s}, \mathbf{Rep}^{s}, \ldots$, we may write $\mathbf E ^{\chi}, \mathbf{Rep}^{\chi}, \ldots$ when we have a parameter $\chi = ( s, \tau )$.

For each $w \in W _n$ and $s \in T _n ( \vec{q} )$, we define $\tau ^s _w$ to be a marked partition adapted to $v_s s$ so that
$$\mathcal O _{\tau ^s _w} \cap {} ^w \mathbb V ^+ _n \cap \mathbb V ^{(v_s s,\vec{q})} _n \subset {} ^w \mathbb V ^+  _n \cap \mathbb V ^{(v_s s,\vec{q})} _n \subset \mathbf E ^{s} ( 0 ) \oplus \mathbf{Rep}^{s}$$
is open dense. It is clear that $\tau ^s _w$ is well-defined up to equivalence (since there are only finitely many $G( s )$-orbits in $\mathbb V ^{(s,\vec{q})}$). We
set $\chi _w := ( v _s s, \tau ^s _w )$ for $\chi = ( s, \tau ) \in \mathsf{P} ( \vec{q} )$. (Note that $\chi _w$ depends on $s$ and $w$, but not on $\tau$.)

For $\chi = ( s, \tau ) \in \mathsf{P} ( \vec{q} )$, we set
$$\rho _{ij} ( \chi ) := \# \{ I \in \chi \mid q_1 q ^i, q_1 q ^j \in I \} \text{ for every } j > i.$$

\begin{theorem}[Abeasis-Del Fra \cite{AD}, Zelevinsky \cite{Z2}]\label{ADZ}
For each $\chi = ( s, \tau ) \in \mathsf{P} ( \vec{q} )$, the collection $\{ \rho _{ij} ( \chi ) \} _{i,j}$ determines $\tau^0$ uniquely. Moreover, we have
\begin{enumerate}
\item $\mathcal O _{( \chi _w ) ^0} \subset \overline{\mathcal O _{\chi ^0}}$ if and only if $\dim A ^{j - i} ( \mathbf E ^s ( i ) ) \le \rho _{ij} ( \chi ) \text{ for every } A \in \mathbf{Rep} _w ^{s}$;
\item If 1) holds, then we have $\mathcal O _{\chi ^0} = \mathcal O _{( \chi _w ) ^0}$ if and only if some $A \in \mathbf{Rep} _w ^{s}$ attains all the equalities in the condition 1). 
\end{enumerate}
Moreover, we have $\mathcal O _{\chi} \subset \overline{\mathcal O _{\chi'}}$ only if $| \chi | = | \chi' |$ and $\rho _{ij} ( \chi ) \le \rho _{ij} (\chi')$ for every $i,j$.
\end{theorem}

\begin{definition}[Elementary modification]\label{em}
Let $\tau = ( \mathbf J, \delta )$ be a marked partition {adapted to $s$}. For $J_1, J_2 \in \mathbf J$, we define another marked partition $\varepsilon_{J_1, J_2} ( \tau ) := ( \mathbf J', \delta' )$ as the maximal marked partition (with respect to the closure ordering) adapted to $s$ which satisfies:
$$\mathbf J ^{\circ} := \mathbf J - \{ J_1,J_2 \} \subset \mathbf J', \delta ( J_1 \cup J_2 ) = \delta' ( \mathbf J' - \mathbf J ^{\circ} ), \text{ and } \mathcal O _{(s,\tau)} \subsetneq \overline{\mathcal O _{(s,\varepsilon_{J_1, J_2} ( \tau ))}}.$$
Since both $\tau$ and $\varepsilon_{J_1, J_2} ( \tau )$ are adapted to $s$, we put $\varepsilon_{\underline{J_1}, \underline{J_2}} ( \chi ) := ( s, \varepsilon_{J_1, J_2} ( \tau ) )$ if $\chi = (s, \tau)$. 
\end{definition}

\begin{lemma}
Keep the setting of Definition \ref{em}. If $\{ J'_1, J'_2 \} = \mathbf J' - \mathbf J^{\circ}$, then we have
\begin{itemize}
\item $\underline{J' _1} = \underline{J _1} \cup \underline{J _2}$ and $\underline{J' _2} = \underline{J _1} \cap \underline{J _2}$ by swapping $J' _1$ and $J' _2$ if necessary;
\item $J' _i$ $(i=1,2)$ is marked if and only if $q_1 \in \underline{J' _i}$ and $\delta ( J _1 \cup J _2 ) = \{0,1\}$.
\end{itemize}
\end{lemma}

\begin{proof}
Straight-forward from Theorem \ref{ADZ} and Definition \ref{em}.
\end{proof}

The following is a reformulation of results from \cite{AD, ADK}:

\begin{theorem}[Abeasis-Del Fra-Kraft]\label{ADK}
For each $\chi = ( s, \tau ) \in \mathsf{P}^0 ( \vec{q} )$ and $I,I' \in \chi$, we set $\chi' := \varepsilon _{I,I'} ( \chi )$ $($Definition \ref{em}$)$. Assume $\chi \neq \chi'$ and $\max I < \max I'$, and we set $I ^{\flat} := I \cap I'$ $($if $I \cap I' \neq \emptyset)$ or $\{ \max I, \min I'\}$ $($if $q \max I = \min I')$.
\begin{enumerate}
\item We have $\dim \mathcal O _{\chi'} = \dim \mathcal O _{\chi} + 1$ if there exists no $I'' \in \chi'$ such that
$$I ^{\flat} \subsetneq I'' \subsetneq I \cup I', \text{ or } I'' = I, I', \{ \max I, \min I' \}.$$
\item If 1)
  holds, then $\overline{\mathcal O _{\chi'}}$ is regular along
  $\mathcal O _{\chi}$ and the defining equation is locally given
  as:
\begin{equation}\label{critsp}
f \in \mathrm{Hom} ( \wedge ^{k} \mathbf E ^{s} ( i ), \wedge ^{k} \mathbf E ^{s} ( j ) )^*,
\end{equation}
where $q_1 q^i = \min I ^{\flat}, q_1 q^j = q \max I ^{\flat}$, and $k := \rho _{ij} ( \chi' )$.
\end{enumerate}
\end{theorem}

We also need the following result.

\begin{lemma}[\cite{CK} Corollary 4.10]\label{ind}
The map
$$K( \mathfrak M_{\vec{q}}^n ) \ni M \mapsto \mathsf{ch} M \in \mathbb Z \left< T _n \right>$$
is an injection. \hfill $\Box$
\end{lemma}

\begin{proposition}\label{indchar}
Let $( \chi _1, \chi _2 ) \in \mathsf{P} ^0 ( \vec{q} ) \times \mathsf{P} ( \vec{q} )$ with $n = | \chi _1 | + | \chi_2 |$. We have
$$\mathsf{ch} M _{\chi _1 \oplus \chi _2} = \mathsf{ch} ( M _{\chi _1} ^{\mathsf A} \boxast M _{\chi _2}  ) = \mathsf{ch} ( {} ^{\mathsf t} M _{\chi _1} ^{\mathsf A} \boxast M _{\chi _2} ).$$
\end{proposition}

\begin{proof}
Let us denote $n_i = | \chi _i |$ ($i=1,2$). We set $\mathbf T := T _{n}$. Let
$\mathbf P \supset B _{n}$ be the parabolic subgroup of $G _{n}$ with
its reductive part $\mathbf L = \mathop{GL} ( n _1 ) \times
\mathop{Sp} ( 2 n _2 )$. Define $W _{\mathbf L} := N _{\mathbf T} ( \mathbf
L ) / \mathbf T \subset W$. We write $\chi _i := (s_i, \tau_i) = (s_i,
\mathbf J ^i, \delta ^i )$ ($i=1,2$), where $\delta ^1 \equiv \{ 0
\}$. We set $\mathbf{v} := \mathbf{v} _{\tau _1 \oplus \tau _2}$ and
$\mathbf{v} _i := \mathbf{v} _{\tau _i}$ for $i=1,2$. We have
$\mathbf{v} = \mathbf{v} _1 \oplus \mathbf{v} _2$. Let $\mathbf r \in
T _{n_1} \cong T _{n_1} \times \{ 1 \} \subset T _{n}$ be the
element such that $\epsilon _i ( \mathbf r ) = r > 1$ for every $1 \le
i \le n _1$ (and $=1$ otherwise). Then, we have $\mathbf r \mathbf{v}
= \mathbf{v}$ and hence $\mathbf r$ acts on $\mathcal E _{\chi _1 \oplus \chi _2}$. 

Here we have $\mathbf{v} _1 \in \mathfrak{gl} _{n_1} = \mathfrak{gl} _{n_1} \oplus \{ 0 \} \subset \mathbb V ^{\mathbf r}$. Let $\mathcal B _{\mathbf{v} _1}$ be the type $\mathsf{A} _{n_1 - 1}$ Springer fiber of $\mathbf{v} _1$. We have
\begin{equation}
( \mathcal E _{\chi _1 \oplus \chi _2} ) ^{\mathbf r} = \bigsqcup _{w \in W^{\vec{n}}} ( ( \mathcal E _{\chi _1 \oplus \chi _2} ) ^{\mathbf r} \cap P \dot{w}^{-1} B / B ) \cong \bigsqcup _{w \in W^{\vec{n}}} \mathcal B _{\mathbf{v} _1} ^{s_1} \times \mathcal E _{\chi _2}.\label{deceS}
\end{equation}
Thanks to \cite{CG} \S 8.2, it follows that each $H _{\bullet} ( ( \mathcal E _{\chi _1 \oplus \chi _2} ) ^{\mathbf r} \cap P \dot{w}^{-1} B / B )$ admits an $R ( \mathbf T )$-module structure with
$$\mathsf{ch} H _{\bullet} ( ( \mathcal E _{\chi _1 \oplus \chi _2} ) ^{\mathbf r} \cap P \dot{w} ^{-1} B / B ) = w \mathsf{ch} H _{\bullet} ( ( \mathcal E _{\chi _1 \oplus \chi _2} ) ^{\mathbf r} \cap P / B ).$$
By \cite{K1} Theorem 6.2, we conclude that
\begin{align*}
& \mathsf{ch} M _{\chi _1 \oplus \chi _2} = \sum _{w \in W^{\vec{n}}} \mathsf{ch} H _{\bullet} ( ( \mathcal E _{\chi _1 \oplus \chi _2} ) ^{\mathbf r} \cap P \dot{w}^{-1} B / B )\\
& = \sum _{w \in W^{\vec{n}}} w \mathsf{ch} H _{\bullet} ( ( \mathcal E _{\chi _1 \oplus \chi _2} ) ^{\mathbf r} \cap P / B ) = \mathsf{ch} ( M _{\chi _1} ^{\mathsf A} \boxast M _{\chi _2} ).
\end{align*}
The case $M _{\chi _1} ^{\mathsf A}$ replaced with ${} ^{\mathsf t} M _{\chi _1} ^{\mathsf A}$ is similar.
\end{proof}

\begin{corollary}\label{indmult}
Keep the setting of Proposition \ref{indchar}. Let $L$ be an irreducible $\mathbb H _{n}$-module. Then, we have
$$[M _{\chi _1 \oplus \chi _2} : L] = [ M _{\chi _1} ^{\mathsf A} \boxast M _{\chi _2} : L]  = [ {} ^{\mathsf t} M _{\chi _1} ^{\mathsf A} \boxast M _{\chi _2} : L].$$
\end{corollary}

\begin{proof}
Combine Proposition \ref{indchar} and Lemma \ref{ind}.
\end{proof}

\begin{lemma}\label{wth}
Let $\chi = ( s, \tau ) \in \mathsf{P} _{n_2} (\vec{q})$. Let $\nu'$ be an integer and set $\nu := m + \nu'$. For every $w_1 \times w _2 \in W [\mathsf{St} ^\nu _{n_1}, \chi ]\cap(\mathfrak S _{n_1} \times W _{n_2})$, we put $w = v_{\mathsf{s} ^\nu _{n_1} \times s} ( w _1 \times w _2 )$ and $\chi' := ( \mathsf{St} ^\nu _{n_1} \oplus
\chi ) _{w}$. Then, we have
\begin{equation}
\rho _{ij} ( \mathsf{St} ^\nu _{n_1} ) + \rho _{ij} ( \chi ) \le \rho _{ij} ( \chi' ) \le \rho _{ij} ( \chi ) + 1\label{mod}
\end{equation}
and $\rho _{ij} ( \chi' ) = \rho _{ij} ( \chi ) + 1$ only if $\nu' \le j < \nu' + n_1$. If we replace $\mathsf{St} ^\nu _{n_1}$ by ${} ^{\mathtt t} \mathsf{St} ^\nu _{n_1}$, then the inequalities (\ref{mod}) remain the same, and $\rho _{ij} ( \chi' ) = \rho _{ij} ( \chi ) + 1$ only if $\nu' \le i < \nu' + n_1$.
\end{lemma}

\begin{proof}
We drop the superscripts $\nu$ during this proof. By construction, the natural $G ( \mathsf s_{n_1} ) \times G ( \chi )$-equivariant embedding
$$\mathbf{Rep} ^{\mathsf{s}_{n_1}} \oplus \mathbf{Rep}^{\chi} \hookrightarrow \mathbf {Rep} ^{\chi'}$$
induces an embedding (of linear spaces which preserves compositions)
$$\mathbf{Rep} ^{\mathsf{s}_{n_1}} _{w_1} \oplus \mathbf{Rep} ^{\chi} _{w_2} \hookrightarrow \mathbf{Rep} ^{\chi'} _{w}.$$
It follows that $\rho _{ij} ( \mathsf{St}_{n_1} ) + \rho _{ij} ( \chi ) \le \rho _{ij} ( \chi' )$ for every $i, j$. Moreover, the condition on $w$ asserts that
\begin{equation}
\Hom ( \mathbf E ^{\mathsf{s}_{n_1}} ( i ), \mathbf E ^{\chi} ( i+1 ) ) \cap \mathbf{Rep} ^{\chi'} _{w} = \{ 0 \} \text{ for every } i \in \mathbb Z.\label{eig}
\end{equation}
It follows that every $A \in \mathbf{Rep} ^{\chi'} _{w}$ preserves $\bigoplus _{i \in \mathbb Z} \mathbf E ^{\chi} ( i ) \subset \bigoplus _{i \in \mathbb Z} \mathbf E ^{\chi'} ( i )$. Moreover, the image of the induced map $\mathbf{Rep} ^{\chi'} _{w} \longrightarrow \mathbf{Rep} ^{\chi}$ is contained in $\mathbf{Rep} ^{\chi} _{w _2}$. Since $\dim \mathbf E ^{\mathsf{s}_{n_1}} ( i ) \le 1$ for every $i \in \mathbb Z$, we conclude $\rho _{ij} ( \chi' ) \le \rho _{ij} ( \chi ) + 1$. This proves the first part of the assertion.

To prove the second assertion, it suffices to see $A ^{j - i} ( \mathbf E ^{\mathsf{s}_{n_1}} (i) ) = \{ 0 \}$ when $j \ge \nu' + n_1$. This follows by (\ref{eig}). For the case $\mathsf{St}$ replaced by ${} ^{\mathtt t} \mathsf{St}$, we apply the same argument except for
$$\Hom ( \mathbf E ^{\chi} ( i ), \mathbf E ^{\bar{\mathsf{s}}_{n_1}} ( i+1 ) ) \cap \mathbf{Rep} ^{\chi'} _{w} = \{ 0 \} \text{ for every } i \in \mathbb Z.$$
instead of (\ref{eig}).
\end{proof}

In view of Lemma \ref{wth}, let us define
\begin{align*}
& \mathcal S _{\nu,n_1} ^{+} ( \chi ) ^{\sim} := \left\{ \chi' = ( s', \tau' ) \mid {\small\begin{matrix}s' = v _{\mathsf s ^\nu _{n_1} \times s} (\mathsf s ^\nu _{n_1} \times s), \rho _{ij} ( \chi' ) - \rho _{ij} ( \chi ) \le 1, \text{ and } \\ \rho _{ij} ( \chi' ) = \rho _{ij} ( \chi ) + 1 \text{ only if } \nu \le j < \nu + n_1 \end{matrix}} \right\}\\
& \mathcal S _{\nu,n_1} ^{-} ( \chi ) ^{\sim} := \left\{ \chi' = ( s', \tau' ) \mid {\small\begin{matrix}s' = v _{\mathsf s ^\nu _{n_1} \times s} (\mathsf s ^\nu _{n_1} \times s), \rho _{ij} ( \chi' ) - \rho _{ij} ( \chi ) \le 1, \text{ and } \\ \rho _{ij} ( \chi' ) = \rho _{ij} ( \chi ) + 1 \text{ only if } \nu \le i < \nu + n_1 \end{matrix}} \right\}.
\end{align*}
Moreover, we define $\mathcal S _{\nu,n_1} ^{\pm} ( \chi ) := \{ \chi' \in \mathcal S _{\nu,n_1} ^{\pm} ( \chi ) ^{\sim} \mid \mathcal O _{\mathsf{St} ^\nu _{n_1} \oplus \chi} \subset \overline{\mathcal O _{\chi'}} \}$.

\begin{proposition}\label{iest}
Let $n = n_1 + n_2$ and $\nu'$ be natural numbers, and set $\nu := m + \nu'$. Let $\chi' \in \mathsf P_{n} ( \vec{q} )$. If we have a surjection
$$\mathsf{St} ^\nu _{n_1} \boxast L _{\chi} \longrightarrow \!\!\!\!\! \rightarrow L _{\chi'} \text{ for some } \chi \in \mathsf P _{n_2} ( \vec{q} ),$$
then we have $\chi' \in \mathcal S _{\nu,n_1} ^+ ( \chi )$. If we replace $\mathsf{St} ^\nu _{n_1}$ with ${} ^{\mathtt t} \mathsf{St} ^\nu _{n_1}$, then the same statement holds only if $\chi' \in \mathcal S _{\nu,n_1} ^- ( \chi )$.
\end{proposition}

\begin{proof}
By the Frobenius reciprocity, it suffices to assume
\begin{equation}
\mathsf{St} ^\nu _{n_1} \boxtimes L _{\chi} \hookrightarrow L _{\chi'}\label{tt}
\end{equation}
as $\mathbb H _{n_1} ^{\mathsf A} \otimes \mathbb H _{n_2}$-modules to deduce $\chi' \in \mathcal S _{\nu,n_1} ^+ ( \chi )$. The condition (\ref{tt}) implies
\begin{equation}
\Psi ( \mathsf{St} ^\nu _{n_1} \boxtimes L _{\chi} ) \subset \Psi ( L _{\chi'} ).\label{prodwt}
\end{equation}
For every $w_1 \times w _2 \in ( \mathfrak S _{n_1} \times W _{n _2} \cap W [ \mathsf{St} ^\nu _{n_1}, \chi ] )$, we have $s' = ( w _1 ^{-1} \mathsf s ^\nu _{n_1} \times w _2 ^{-1} s ) ^{-1} \in \Psi ( \mathsf{St} ^\nu _{n_1}
\boxtimes L _{\chi} )$. We set $\widetilde{\chi} := ( \mathsf{St} ^\nu _{n_1} \oplus
    \chi ) _{v _{s'} ( w_1 \times w_2 )}$. By construction (cf. \cite{CK} \S 2.1), we have
$$\xi \in \mathsf{P} _n ( \vec{q}) \text{ satisfies } L _{\xi} [ s' ] \neq \{ 0 \} \Rightarrow \mathcal O _{\xi} \subset \overline{\mathcal O _{\widetilde{\chi}}}.$$
Hence, we conclude $\mathcal O _{\chi'} \subset \overline{\mathcal O _{\widetilde{\chi}}}$. By Lemma \ref{wth}, this happens only
if $\chi' \in \mathcal S _{\nu,n_1} ^+ ( \chi ) ^{\sim}$. By Corollary \ref{indmult}, we have necessarily $[ M _{ \mathsf{St} ^\nu _{n_1} \oplus \chi} : L _{\chi'} ] \neq 0$ provided that {(\ref{tt}) holds.} This implies $\mathcal O _{\mathsf{St} ^\nu _{n_1} \oplus \chi} \subset \overline{\mathcal O _{\chi'}}$, hence we conclude $\chi' \in \mathcal S _{\nu,n_1} ^{+} ( \chi )$. The other case is completely analogous.
\end{proof}

\section{Delimits of tempered modules}\label{delimits}
In this section, we fix $m_0 \in \frac{1}{2} \mathbb Z$ and assume $|m
- m_0| < \frac{1}{2}$. Let $\sigma$ be a partition of $n$. As
mentioned in the introduction, for every such $m$ there is a central
character $\mathsf{c} _m ^\sigma \in T _n ( \vec{q} )$ and a single
discrete series with parameter $\mathsf{ds} _m ( \sigma ) = (
\mathsf{c}_m ^{\sigma}, ds _m (\sigma) )$; here $ds _m (\sigma)$ is a
marked partition adapted to $\mathsf{c} _m ^\sigma$ which
parameterizes this discrete series. It might be helpful to visualize $\mathsf{ds} _m ( \sigma )$ combinatorially as a
left justified decreasing Young diagram coming from the partition
$\sigma$ by labeling every box with its $m$-content, i.e., the number
$m+c(i,j)$, where $c(i,j)$ is the content $i-j$ of the box in the
$(i,j)$ position. 

We sometimes identify $\mathsf{ds} _m ( \sigma )$ with $L _{\mathsf{ds} _m ( \sigma )}$.
 Let $\mathsf{mp} _m ( \sigma )$ denote the anti-spherical parameter
 with central character $\mathsf{c} _m ^\sigma$. The $G(\mathsf
 c_m^\sigma)$-orbit indexed by $\mathsf{mp} _m ( \sigma )$ is open
 {dense}. 

\subsection{Tempered delimits at generic parameter}

We call a parameter $\chi\in\mathsf{P} ( \vec{q} )$ positive (resp. negative) if we
have $\mathtt{E} _m ( I ) > 1$ $($resp. $< 1)$ for every $I \in
\chi$.

\begin{theorem}[\cite{CK} \S 3.3 + \cite{K1} Theorem 7.4]
The parameter $\mathsf{ds} _m ( \sigma )$ admits a unique decomposition $\mathsf{ds} _m ( \sigma ) = \mathsf{ds} _m ^+ ( \sigma ) \oplus \mathsf{ds} _m ^- ( \sigma )$ such that $\mathsf{ds} _m ^+ ( \sigma )$ is positive and $\mathsf{ds} _m ^- ( \sigma )$ is negative. Moreover,
\begin{enumerate}
\item $\mathsf{ds} _m ^- ( \sigma )$ is not marked, and hence we can regard $\mathsf{ds} _m ^- ( \sigma ) \in \mathsf Q ( q )$;
\item two parameters $\mathsf{ds} _m ^+ ( \sigma )$ and $\mathsf{ds} _m ^- ( \sigma )$ are nested to each other;
\item we have a surjection $( {} ^{\mathsf t} L ^{\mathsf A} _{\mathsf{ds} _m ^- ( \sigma )} \boxast L _{\mathsf{ds} _m ^+ ( \sigma )} ) \longrightarrow \!\!\!\!\! \rightarrow L _{\mathsf{ds} _m ( \sigma )}$.
\end{enumerate}
\end{theorem}

We denote by $ds _m ^{\pm} (\sigma)$ the marked partitions corresponding to $\mathsf{ds} _m ^{\pm} ( \sigma )$, respectively.

\begin{definition}[Delimits of tempered modules]\label{d:tdl}
An algebraic flat family of irreducible $\mathbb H _{n,m}$-modules $L ^m$ depending on $m$ ($m_0 - \frac{1}{2} < m < m_0 + \frac{1}{2}$ but $m \neq m_0$) with central character $\mathsf{c} _m ^\sigma$ is called a delimit of tempered module, or just a tempered delimit if the limit $\lim _{m \to m_0} L^m$ is a tempered module. Let $\mathcal{D} _{m_0} ( \sigma )$ be the set of isomorphism classes of irreducible tempered delimits with central character $\mathsf{c} _m ^\sigma$.
\end{definition}

A segment $I$ is called balanced along $m _0$ if $e _+ ( I ) e_- ( I ) = q ^{2 ( m - m_0 )}$. A multisegment $\mathbf I$ is called balanced along $m_0$ if each member is a balanced segment along $m_0$.

Below, for a marked partition $\tau \in \mathcal{D} _{m_0} ( \sigma )$
and a marked partition $\tau'$ obtained from $\tau$, we denote the
corresponding parameters by using bold letters. (I.e. ${\boldsymbol
  \tau} = ( \mathsf c _m ^\sigma, \tau)$, ${\boldsymbol \tau'}$,
etc...).

\begin{lemma}\label{tli}
Let $\tau_m$ denote a marked partition adapted to every $\mathsf{c} _m
^\sigma$ $($with $m_0 < m < m_0 + \frac{1}{2})$. Then
${\boldsymbol\tau} _m$ is the parameter of a tempered delimit if and only if there exists a multisegment ${\boldsymbol \tau} _m ^{\mathtt e}$ and a marked partition $\tau _m ^{\mathtt s}$ such that the following conditions hold:
\begin{itemize}
\item We have a surjective map $( {} ^{\mathsf t} L _{{\boldsymbol \tau} _m ^{\mathtt e}} ^{\mathsf A} \boxast L _{{\boldsymbol \tau} _m ^{\mathtt s}} ) \longrightarrow \!\!\!\!\!\!\!\!\! \longrightarrow L _{{\boldsymbol \tau} _m}$;
\item ${\boldsymbol \tau} _m ^{\mathtt s} = \mathsf{ds} _{m} ( \sigma' )$ for some partition $\sigma'$ obtained by removing $\# {\boldsymbol \tau} _m ^{\mathtt e}$ hooks from $\sigma$ $($as Young diagrams$)$;
\item ${\boldsymbol \tau} _m ^{\mathtt e}$ is a balanced multisegment along $m_0$;
\item We have $e _+ ( I ) \neq e _+ ( I' )$ and $e _- ( I ) \neq e _- ( I' )$ for every pair $I, I'$ of segments in ${\boldsymbol \tau} _m ^{\mathtt e}$.
\end{itemize}
Entirely the same statement holds if we replace $m$ with $m_0 - \frac{1}{2} < m' < m_0$ and ${} ^{\mathsf t} L _{{\boldsymbol \tau} _m ^{\mathtt e}} ^{\mathsf A}$ with $L _{{\boldsymbol \tau} _{m'} ^{\mathtt e}} ^{\mathsf A}$.
\end{lemma}

\begin{proof}
We prove only the case of ${} ^{\mathsf t} L _{{\boldsymbol \tau} _m
    ^{\mathtt e}} ^{\mathsf A}$ since the case of $L _{{\boldsymbol
      \tau} _{m'} ^{\mathtt e}} ^{\mathsf A}$ is completely
  analogous. Let $\tau$ be a marked partition corresponding to a
tempered delimit. By the Evens-Langlands classification \cite{Ev}, we have a unique quotient map
\begin{equation}
L ^{\mathsf A} \boxast L _{{\boldsymbol \tau} _m ^{\mathtt s}} \longrightarrow \!\!\!\!\!\!\!\!\! \longrightarrow L _{\boldsymbol \tau_m},\label{ELs}
\end{equation}
where $L _{{\boldsymbol \tau} _m ^{\mathtt s}}$ is a discrete series at $m_0 < m < m_0 + \frac{1}{2}$ for some smaller affine Hecke algebra of type $\mathsf C$. Here (\ref{ELs}) is a priori surjection for a specific value of $m$. Taking account into the fact that $\mathsf{P} (\vec{q})$ is constant for all the generic value of $m$ and the Morita equivalences from \cite{K1} \S 9, we deduce that each element of $\mathsf{P} (\vec{q})$ defines an algebraic flat family of representations depending on $m$ such that the function $\mathsf{ch}$ is continuous on $m$. Thanks to the uniqueness of quotients of Evens-Langlands induced modules (for each individual value), we deduce that (\ref{ELs}) prolongs to an algebraic family depending on $m$ and its quotient is irreducible for all values of $m_0 < m < m_0 + \frac{1}{2}$.

Assume $\tau _m ^{\mathtt s} = ( \mathbf J, \delta )={ds}_m(\sigma')$, for some smaller partition $\sigma'$. The type $\mathsf A$ factor $L ^{\mathsf A}$ is the unique quotient of the induction of
$$\mathsf{St} _{k _1} ^{\nu _1'} \boxtimes \mathsf{St} _{k _2} ^{\nu _2'} \boxtimes \cdots \boxtimes \mathsf{St} _{k _p} ^{\nu _p'} \text{ with } {\nu_1' - \frac{k_1}{2} \le \nu _2' - \frac{k_2}{2}\le \cdots \le \nu _p' - \frac{k_p}{2} \le \frac{-1}{2}},$$
 to $\mathbb H _{k} ^{\mathsf A}$, where $k = \sum_{i=1}^p k_i$. By the Frobenius reciprocity, we have
\begin{equation}
\mathsf{St} _{k _1} ^{\nu _1'} \boxtimes \mathsf{St} _{k _2} ^{\nu _2'} \boxtimes \cdots \boxtimes \mathsf{St} _{k _p} ^{\nu _p'} \boxtimes L _{{\boldsymbol \tau} _m ^{\mathtt s}} \subset L _{{\boldsymbol \tau_m}}\label{levi}
\end{equation}
as $R ( T _n )$-modules. (Here $\nu _1', \nu _2', \ldots,$ are a priori
real numbers.) In particular, the inclusion gives a nonzero map at the limit $m
\to m_0$ {(c.f. \cite{CK} \S 2.4)}. Therefore, we need $\lim _{m \to m_0} \nu _i' =
( 1 - k _i ) / 2$ for $i=1,\ldots,p$ by the temperedness of the weights coming from
the left hand side of (\ref{levi}). In particular, either $\mathsf{St} _{k _i} ^{\nu _i'}$ or ${}^{\mathtt t}\mathsf{St} _{k _i} ^{\nu _i'}$ corresponds to a $m_0$-balanced segment for each $1 \le i \le p$.

Since we fixed the central character to be $\mathsf{c} _m ^\sigma$, we conclude that $\nu _i' = ( m _0 - m ) + \frac{1 - k_i}{2}$ for every $i = 1,2,\ldots,p$. By induction-by-stages, we conclude
$${} ^{\mathtt t} \mathsf{St} _{k _1} ^{\nu _1} \boxast ( {} ^{\mathtt t} \mathsf{St} _{k _2} ^{\nu _2} \boxast \cdots ( {} ^{\mathtt t} \mathsf{St} _{k _p} ^{\nu _p} \boxast L _{{\boldsymbol \tau} _m ^{\mathtt s}} )) \longrightarrow \!\!\!\!\!\!\!\!\! \longrightarrow L _{{\boldsymbol \tau} _m}$$
with $\nu _i = ( m - m _0 ) + \frac{1 - k_i}{2}$ for $i=1,\ldots,
p$. In particular, the multisegment ${\boldsymbol \tau} _m ^{\mathtt
  e} := \{ \mathsf{St} _{k _1} ^{\nu _1}, \mathsf{St} _{k _2} ^{\nu
  _2}, \ldots, \mathsf{St} _{k _p} ^{\nu _p} \}$ satisfies the first
and the third conditions.

\begin{claim}\label{hooks}
The partition $\sigma'$ is obtained by removing a certain number of hooks from $\sigma$ of length
$k_1,k_2,\ldots$. In particular, $k_1, k_2, \ldots, k _p$ are distinct.
\end{claim}
\begin{proof}
Let $\widetilde\sigma=\{I_1,I_2,\dots,I_\ell\}$ be a multisegment viewed as $\sigma$ ``unbent'' along the diagonal. Notice that $q_1\in I_1\Subset I_2\Subset\dots\Subset I_\ell$. A hook in the original partition
  $\sigma$ then becomes a path consisting of the union of a segment of
the form $\{q_1,q_1 q,\dots,e_+(I_i)\}$ with a segment of the form
$\{e_-(I_j),qe_-(I_j),\dots,q^{-1}q_1\}$ for some $I_i,I_j$. A
subpartition $\widetilde\sigma'\subset\widetilde\sigma$ must satisfy
$\widetilde\sigma'=\{I'_{i_1},\dots,I'_{i_p}\},$ $q_1\in
I_{i_1}\Subset\dots\Subset I_{i_p}$, and $I'_{i_t}\subset I_{i_t}.$
It is sufficient (by induction) to show that there exists a partition $\widetilde\sigma''$ such
that $\widetilde\sigma'\subset \widetilde\sigma''\subset \widetilde\sigma$, and $\sigma''$ is
obtained from $\sigma$ by removing one (balanced) hook.
  Set $\mathbf S=\widetilde\sigma\setminus\widetilde\sigma'$, and
  regard it as a balanced multisegment. We denote $S = \{ b \in I \mid I \in \mathbf S\}$ and $S^{-1}=\{b^{-1}:b\in S\}$. Find the largest value $b_{\max}$ in $S\cup S^{-1}.$ Assume this is in $S$ (the
other case is completely analogous). Then $b_{\max}=e_+(I_i)$ for some
$I_i\in\widetilde\sigma.$ Next find
the smallest value in $S$, denote it $b_{\min}.$ (Notice that
$b_{\min}^{-1}$ is the largest values in $S^{-1}$.) Similarly, we must
have $b_{\min}=e_-(I_j)$ for some $I_j\in \widetilde\sigma.$ We
claim that the segment $\{b_{\min},\dotsc,b_{\max}\}$, which is a hook, belongs to $\mathbf S$, or else  there exist $b_-,b_+\in \{b_{\min},\dotsc,b_{\max}\}\setminus S$ such that
$\{b_{\min},\dotsc,q^{-1}b_-\} \in \mathbf S$ and $\{b_+q,\dotsc,b_{\max}\} \in \mathbf S.$ But
then it is clear that the segment $\{b_+q,\dotsc,b_{\max}\}$ cannot be balanced.
\end{proof}

We return to the proof of Lemma \ref{tli}.

Thanks to Claim \ref{hooks}, we deduce the second and the fourth conditions. Therefore, we have proved the ``only if" part of the assertion.

We prove the ``if'' part of the assertion. We recall that $L _{{\boldsymbol \tau} _m ^{\mathtt e}} ^{\mathsf A}$, $L _{{\boldsymbol \tau} _m ^{\mathtt s}}$, (and hence $L _{{\boldsymbol \tau} _m ^{\mathtt e}} ^{\mathsf A} \boxast L _{{\boldsymbol \tau} _m ^{\mathtt s}}$) are algebraic families depending on $m$. Moreover, both of $L _{{\boldsymbol \tau} _m ^{\mathtt s}}$ and $\lim _{m \to m_0} L _{{\boldsymbol \tau} _m ^{\mathtt e}} ^{\mathsf A}$ are tempered modules by assumptions 1) and 2), respectively. In particular, $\lim _{m \to m_0} ( L _{{\boldsymbol \tau} _m ^{\mathtt e}} ^{\mathsf A} \boxast L _{{\boldsymbol \tau} _m ^{\mathtt s}} )$ is tempered, and hence all of its irreducible constituents are tempered.
\end{proof}

\begin{corollary}[of the proof of Lemma \ref{tli}]\label{countd}
Keep the setting of Lemma \ref{tli}. Let $h$ be the number of hooks in
$\sigma$ which give balanced segments along $m _0$. Then, we have
$\# \mathcal D _{m_0} ( \sigma ) = 2 ^h$. \hfill $\Box$
\end{corollary}

\subsection{The classification of tempered delimits}\label{tdc}

We define two subsets of $\mathcal D _{m_0} ( \sigma )$ as:
$$\mathcal D _{m_0} ^{\pm} ( \sigma ) := \{ \tau \in \mathcal D _{m_0} ( \sigma ) \mid {\boldsymbol \tau} _m ^{\mathtt s} = \mathsf{ds} _m ^{\pm} ( \sigma' ) \text{ for } \pm m _0 - \frac{1}{2} < \pm m < \pm m_0\},$$
where the $\pm$ denote a uniform choice of $+$ or $-$, and $\sigma'$ is borrowed from Lemma \ref{tli}.

\begin{proposition}\label{cd}
Let $\tau _m ^{\mathtt e} = ( \mathbf J ^{\mathtt e}, 0 )$ and $\tau _m ^{\mathtt s}$ be the marked partitions obtained from $\tau_m \in \mathcal D _{m_0} ( \sigma )$ by Lemma \ref{tli}. Then there exists a marked partition $\tau _m ^{\mathtt f} = ( \mathbf J ^{\mathtt e}, \delta ^{\mathtt f} )$ $($obtained from $\tau _m ^{\mathtt e}$ by possibly changing the marking$)$ and a decomposition $\tau_m = \tau _m ^{\mathtt f} \oplus \tau _m ^{\mathtt s}$ if one of the following conditions hold:
\begin{itemize}
\item $\tau \in \mathcal D _{m_0} ^{+} ( \sigma )$ and $m_0 - \frac{1}{2} < m < m_0$;
\item $\tau \in \mathcal D _{m_0} ^{-} ( \sigma )$ and $m_0 < m < m_0 + \frac{1}{2}$.
\end{itemize}
\end{proposition}

\begin{proof}
Since the two cases are completely analogous, we prove only the first case. We use the intermediate step of the $\mathsf{ds} _m$-algorithm (in the sense of \cite{CK} Algorithm 3.3 step 2) to $\sigma$. It yields a sequence of segments
\begin{eqnarray}
I_1, I_2, I_3, \ldots \text{ with } \max \{ e _+ ( I_1 ), e _- ( I_1 ) ^{-1} \}> \max \{ e _+ ( I_2 ), e _- ( I_2 ) ^{-1} \} > \cdots\label{iout}
\end{eqnarray}
so that each segment of $\mathsf{ds} _m ( \sigma )$ is a union of at
most two of them. 

Let $\mathbf I ^{\circ}$ be the set of all segments $I$ such that $I$
is obtained by gluing $I ^+$ and $I ^-$ in (\ref{iout}) with the
property that $e _+ ( I ^+ ) = q ^{2 (m - m _0)} e _- ( I ^- )
^{-1}$. Let $\tau _m ^{\mathtt e}$ and $\tau _m ^{\mathtt s}$ be the
marked partitions from Lemma \ref{tli}. We have ${\boldsymbol \tau} _m
^{\mathtt e} \subset \mathbf I ^{\circ}$. For each $I \in \mathbf I
^{\circ}$, we have some $j ( I )$ so that $I = I _{j ( I )} \cup I _{j
  ( I ) + 1}$. Notice that $e _+ ( I _{j ( I )} ) < e _+ ( I _{j ( I )
  + 1} )$ by $m < m_0$. The assumption $\mathsf{ds} _{m} ( \sigma' ) =
\mathsf{ds} _{m} ^+ ( \sigma' )$ implies that we have $\mathtt E _m (
I' ) > 1$ for every $I' \in \mathsf{ds} _{m} ( \sigma' )$. In
particular, there exists no $I' \in \mathsf{ds} _{m} ( \sigma' )$ such
that $I' \triangleleft I$.

Moreover, we have $I \Subset I''$ or $I'' \Subset I$ for every distinct $I, I'' \in \mathbf J ^{\mathtt e}$. We set ${\boldsymbol \tau} _m ^{\mathtt e} = \{ I _1 ^{\mathtt e}, \ldots, I _N  ^{\mathtt e}\}$ and ${\boldsymbol \tau} _m ^{\mathtt s, k} := \{ I _1 ^{\mathtt e} \} \oplus \{ I _2 ^{\mathtt e} \} \oplus \cdots \oplus \{ I _k ^{\mathtt e} \} \oplus {\boldsymbol \tau} _{m} ^{\mathtt s}$. Applying Proposition \ref{iest} to $\{ I _{k-1} ^{\mathtt e} \}$ and ${\boldsymbol \tau} _m ^{\mathtt s, k}$ for each $k=1,\ldots, N$, we conclude that
$$L _{I_1 ^{\mathtt e}} ^{\mathsf A} \boxast ( \cdots ( L _{I_N ^{\mathtt e}} ^{\mathsf A} \boxast L _{{\boldsymbol \tau} _m ^{\mathtt s}} )) \longrightarrow \!\!\!\!\!\!\!\!\! \longrightarrow L _{{\boldsymbol \tau} _m ^{\mathtt e}} ^{\mathsf A} \boxast L _{{\boldsymbol \tau} _m ^{\mathtt s}} \longrightarrow \!\!\!\!\!\!\!\!\! \longrightarrow L _{{\boldsymbol \tau} _m ^{\mathtt f} \oplus {\boldsymbol \tau} _m ^{\mathtt s}}$$
for some ${\boldsymbol \tau} _m ^{\mathtt f}$ obtained from ${\boldsymbol \tau} _m ^{\mathtt e}$ by changing the markings if necessary. By the uniqueness of the quotient of the Evens-Langlands induction, we conclude that ${\boldsymbol \tau} _m = {\boldsymbol \tau} _m ^{\mathtt f} \oplus {\boldsymbol \tau} _m ^{\mathtt s}$ as required.
\end{proof}

\begin{lemma}\label{ndpm}
There exists a unique decomposition $n=n_1+n_2+\dots+n_p$, and a unique sequence $\sigma ^1, \sigma ^2,
  \cdots, \sigma ^p$ of partitions of $n_1,n_2,\dots, n_p$, with the following properties:
\begin{enumerate}
\item $I \Subset I'$ if $I \in \mathsf{mp} _m ( \sigma ^i )$ and $I' \in \mathsf{mp} _m (\sigma^j)$ for $i < j$;
\item $\mathsf{mp} _m ( \sigma ^i )$ is positive if and only if $\mathsf{mp} _m ( \sigma ^{i+1} )$ is negative for each $i$;
\item $\mathsf{mp} _m ( \sigma ^i )$ is negative if and only if $\mathsf{mp} _m ( \sigma ^{i+1} )$ is positive for each $i$.
\end{enumerate}
\end{lemma}

\begin{proof}
Let $\sigma ^1 _0, \sigma ^2 _0, \cdots, \sigma _0 ^{p'}$ be the sequence of partitions which gives the finest nested component decomposition
$$\mathsf{mp} _m ( \sigma ) = \mathsf{mp} _m ( \sigma ^1 _0 ) \oplus \mathsf{mp} _m ( \sigma ^2 _0 ) \oplus \cdots \oplus \mathsf{mp} _m ( \sigma ^{p'} _0 ).$$
By rearranging the order of the sequence if necessary, we can assume $I \Subset I'$ if $I \in \mathsf{mp} _m ( \sigma ^i_0 )$ and $I' \in \mathsf{mp} _m (\sigma^j _0)$ for $i < j$. From \cite{CK}, \S 3.5, we know that
  every nested component $\mathsf{mp} _m ( \sigma ) = \mathsf{mp} _m (
  \sigma ^i )$ is either positive or negative. Hence, by joining $\sigma _0 ^i$ with $\sigma _0 ^{i+1}$ if both of $\mathsf{mp} _m ( \sigma ^i_0 )$ and $\mathsf{mp} _m ( \sigma ^{i+1}_0 )$ are simultaneously positive or negative, we obtain the desired sequence of partitions.
\end{proof}

We refer the decomposition of $\sigma$ into $\sigma ^1, \sigma^2, \ldots, \sigma^p$ in Lemma \ref{ndpm} as the canonical decomposition of $\sigma$ (with respect to the parameter $m$).

\begin{lemma}\label{ndbs}
Fix the canonical decomposition $\sigma ^1, \sigma ^2, \cdots, \sigma ^p$ of $\sigma$. Let $h _i$ be the number of balanced hooks of $\sigma ^i$ and let $h$ be the number of balanced hooks of $\sigma$ $($along $m_0)$. Then, we have $h = \sum _i h_i$.
\end{lemma}

\begin{proof}
Let $I$ be a balanced segment obtained from a hook of $\sigma$. Then,
there exists $I ^+, I^- \in \mathsf{mp} _m ( \sigma )$ such that $e _-
( I ) = e _- ( I ^- )$ and $e _+ ( I ) = e _+ ( I ^+ )$. If $I^+ =
I^-$, then we have $I ^+ = I^- = I$. Hence, we have $I \in \mathsf{mp}
_m ( \sigma ^i )$ for some $i$. Otherwise, we have $\mathtt{E} _{m_0} ( I
^{+} ) \neq 1 \neq \mathtt{E} _{m_0} ( I ^{-} )$. To prove the assertion,
it suffices to verify $I^+, I^- \in \mathsf{mp} _m ( \sigma ^i )$ for some $i$. We assume to the contrary to deduce contradiction. Then, we
have either $I ^+ \Subset I^-$ or $I ^- \Subset I^+$. By inspection, we see that if $I ^+ \Subset I^-$, then both $I^+, I^-$
  are positive. Similarly, if $I ^- \Subset I^+$, then both are negative.
Assume that $I ^+ \Subset I^-$ (and hence they are positive). If there exists $I' \in \mathsf{mp} _m ( \sigma )$ such that $I ^+ \Subset I' \Subset I^-$, then we have $e _+ ( I' ) > e _+ ( I )$ and $e _- ( I' ) > e _- ( I )$. It follows that $I'$ is automatically positive. Therefore, we conclude that $I^+, I^- \in \mathsf{mp} _m ( \sigma ^i )$ for some $i$ by construction of $\sigma ^i$.
We have $I^+, I^- \in \mathsf{mp} _m ( \sigma ^i )$ for some $i$ in the case $I ^- \Subset I^+$ by a similar argument, which completes the proof.
\end{proof}

\begin{definition}\label{Cdef}
Let $\sigma ^1, \sigma^2, \ldots, \sigma ^p$ be the canonical decomposition of $\sigma$ with respect to the parameter $m_0 < m < m_0 + \frac{1}{2}$. Let $\mathcal C _{m_0} ( \sigma ^k )$ ($1 \le k \le p$) be the set of (equivalence class of) marked partitions $\tau_k$ adapted to $\mathsf{c} ^{m} _{\sigma ^k}$ which admit a decomposition
\begin{equation}
\tau_k = \tau ^{\sharp} _k \oplus \tau ^{\flat} _k \oplus \tau ^{+} _k \oplus \tau ^{-} _k = \tau ^{\sharp} _k \oplus \tau ^{\perp} _k \label{td-dec-k}
\end{equation}
$(\tau ^{\perp} _k = \tau ^{\flat} _k \oplus \tau ^{+} _k \oplus \tau ^{-} _k )$ with the following properties:
\begin{enumerate}
\item $\tau ^{\sharp} _k$ is the set of all unmarked $J \in \tau_k$ such that $\underline{J}$ is a balanced segment along $m _0$ obtained from a hook of $\sigma ^k$;
\item $\tau ^{\flat} _k$ is the set of all marked $J \in \tau_k$ such that $\underline{J}$ is a balanced segment along $m _0$ obtained from a hook of $\sigma ^k$. We have $\tau ^{\flat} _k = \emptyset$ if $\mathsf{mp} ( \sigma ^k )$ is negative;
\item $\tau ^+ _k = ds _{m'} ( \sigma ^+ ) = ds _{m'} ^- ( \sigma ^+ )$ for some $\sigma ^+$ and $m_0 < m' < m_0 + \frac{1}{2}$;
\item $\tau ^- _k = ds _{m'} ( \sigma ^- ) = ds _{m'} ^+ ( \sigma ^- )$ for some $\sigma ^-$ and $m_0 - \frac{1}{2} < m' < m_0$.
\end{enumerate}
Then, we define
\begin{equation}\label{Csig}
\mathcal C _{m_0} ( \sigma ) :=\{ \bigoplus _{k=1} ^p \tau _k \mid \tau _k \in \mathcal C _{m_0} ( \sigma ^k ) \text{ for } k=1,\ldots,p\}.
\end{equation}
We may sometimes identify $\tau \in \mathcal C _{m_0} ( \sigma )$ with the corresponding parameter, which we denote by ${\boldsymbol \tau} := ( \mathsf{c} ^{m} _{\sigma}, \tau )$.
\end{definition}

\begin{lemma}\label{Cbasic}
Keep the setting of Definition \ref{Cdef}. For each $\tau _j \in \mathcal C _{m_0} ( \sigma ^j )$ and $\tau _l \in \mathcal C _{m_0} ( \sigma ^l )$ with $j \neq l$, the two parameters ${\boldsymbol \tau} _j$ and ${\boldsymbol \tau} _l$ are nested to each other. In addition, each $\tau \in \mathcal C _{m_0} ( \sigma )$ admits a decomposition
\begin{equation}
\tau = \tau ^{\sharp}  \oplus \tau ^{\flat} \oplus \tau ^{+} \oplus \tau ^{-} = \tau ^{\sharp} \oplus \tau ^{\perp}\label{td-dec}
\end{equation}
with the properties 1)--4) $(\tau _k$ replaced with $\tau)$. 
\end{lemma}

\begin{proof}
Since all balanced segments are obtained by gluing the intermediate output of the $\mathsf{ds} _m$-algorithm as in the proof of Proposition \ref{cd}, we deduce
\begin{align*}
& \max \{ e _+ ( I ) \mid I \in \mathsf{mp} ( \sigma ^j ) \} = \max \{ e _+ ( I ) \mid I \in {\boldsymbol \tau} _j \} \text{ and }\\
& \min \{ e _- ( I ) \mid I \in \mathsf{mp} ( \sigma ^j ) \} = \min \{ e _- ( I ) \mid I \in {\boldsymbol \tau} _j \}.
\end{align*}
If $j > 1$, then we further have
\begin{align*}
& \{ e _+ ( I ) \mid I \in \mathsf{mp} ( \sigma ^j ) \} = \{ e _+ ( I ) \mid I \in {\boldsymbol \tau} _j \} \text{ and }\\
& \{ e _- ( I ) \mid I \in \mathsf{mp} ( \sigma ^j ) \} = \{ e _- ( I ) \mid I \in {\boldsymbol \tau} _j \}.
\end{align*}
Therefore, we deduce
$$\{ e _+ ( I ) \mid I \in {\boldsymbol \tau} _j \} > \{ e _+ ( I ) \mid I \in {\boldsymbol \tau} _l \} \text{ and } \{ e _- ( I ) \mid I \in {\boldsymbol \tau} _j \} < \{ e _- ( I ) \mid I \in {\boldsymbol \tau} _l \}$$
whenever $j > l$. This implies the first assertion. Thanks to Lemma \ref{ndpm} and \cite{CK} Corollary 3.17, the second assertion follows from the first assertion. 
\end{proof}

\begin{corollary}\label{DinC}
We have $\# \mathcal D _{m_0} ( \sigma ) \le \# \mathcal C _{m_0} (\sigma)$.
\end{corollary}

\begin{proof}
Thanks to Lemma \ref{ndbs}, Definition \ref{Cdef} and Lemma
\ref{Cbasic}, it suffices to prove the assertion only {when} $\sigma =
\sigma ^1$ gives the canonical decomposition of $\sigma$ with respect
to $m_0 < m < m_0 + \frac{1}{2}$. Thus, we can assume that
$\mathsf{mp} ( \sigma )$ is either positive or negative. If
$\mathsf{mp} ( \sigma )$ is positive, then we deduce $\mathcal D
_{m_0} ( \sigma ) \subset \mathcal C _{m_0} (\sigma)$ by Proposition
\ref{cd}.

If $\mathsf{mp} ( \sigma )$ is negative, then we have $e_- ( I ) ^{-1} \ge q ^{1-\epsilon}e_+ ( I )$ for $\epsilon = 2 ( m - m_0 ) > 0$. We borrow notation $I_1, I_2, \ldots$ and $\mathbf I ^{\circ}$ from the proof of Proposition \ref{cd}. If $I \in \mathbf I ^{\circ}$ is obtained as a union of $I _{j ( I )}$ and $I _{j ( I ) + 1}$, then $I _{j ( I )}$ is glued with some $I _{k}$ with $k < j ( I )$. Therefore, we deduce that $\mathsf{ds} _m ( \sigma' ) = \mathsf{ds} _m ^- ( \sigma' )$ implies that $\mathsf{ds} _m ( \sigma' )$ does not contain a balanced segment along $m_0$. It follows that $\# \mathcal C _{m_0} ( \sigma )$ is at least the cardinality of the power set of $\mathbf I ^{\circ}$. Hence, we conclude the result by Corollary \ref{countd} in this case.
\end{proof}

\begin{definition}\label{standard}
The decomposition (\ref{td-dec}) is unique if we rearrange the cardinality of $\tau ^{\sharp}$ to be maximal in the equivalence class (in the sense of (\ref{equiv})). When this maximality condition is attained, we call (\ref{td-dec}) the standard decomposition (with respect to the parameter $m_0$).
\end{definition}

\begin{proposition}\label{cr}
Let $\tau, \tau' \in \mathcal{C} _{m _0} ( \sigma )$. We have
$\mathcal O _{{\boldsymbol \tau}'} \subset \overline{\mathcal O _{{\boldsymbol \tau}}}$ if and
only if ${\boldsymbol \tau} ^{\sharp} \subset ( {\boldsymbol \tau}' ) ^{\sharp}$ as
multisegments {$($by using standard decomposition$)$}. Moreover, we have
$$\dim \mathcal O _{{\boldsymbol \tau}} = \dim \mathcal O _{{\boldsymbol \tau}'} + \# ( {\boldsymbol \tau}' ) ^{\sharp} - \#  {\boldsymbol \tau} ^{\sharp}.$$
\end{proposition}

\begin{proof}
Thanks to Lemma \ref{Cbasic} and \cite{CK} Corollary 2.10, it suffices to prove the assertion in the case $\mathsf{mp} _m ( \sigma )$ consists of only positive (or only negative) nested components. By separating out nested components which satisfies $\ds _{m'} ( \sigma' ) = \ds _{m} ( \sigma' )$ for $m_0 - \frac{1}{2} < m' < m_0 <  m < m_0 + \frac{1}{2}$, there remain three cases to be considered: {\bf 0)} $\mathsf{mp} _m ( \sigma )$ consists of a unique balanced segment, or
\begin{align}\nonumber
& {\bf p)} & e _+ ( I ) \ge q ^{1 + \epsilon} e _- ( I ) ^{-1}, \text{ for every } I \in \mathsf{mp} _m ( \sigma ), & \text{ or }\\
& {\bf n)} & e _+ ( I ) \le q ^{-1 + \epsilon} e _- ( I ) ^{-1} \text{ for every } I \in \mathsf{mp} _m ( \sigma ), &\label{pn}
\end{align}
where $\epsilon = 2 ( m - m _0 ) > 0$.

\item Case {\bf 0)}: We have $\tau = ( \{ J \}, \delta )$ with a single $J$ and different choice of $\delta$ from $\tau'$. If $\delta \not\equiv 0$, then we have $\tau = \tau ^{\flat}$ and if $\delta \equiv 0$, then we have $\tau = \tau ^{\sharp}$. Hence, the assertion is straight-forward in this case. 

\item Case {\bf p)}: We assume that $
 {\boldsymbol \tau} ^{\sharp} \cup \{ I \} = ( {\boldsymbol \tau}' ) ^{\sharp}$ as multisegments. By the $\mathsf{ds} _m$ algorithm and condition {\bf p)}, we deduce
  that there exists $I _* \in \boldsymbol \tau$ so that $I \triangleleft I _*$. 
 By rearranging $I_*$ if necessary, we can assume that $e _- ( I _* ) = \min \{ e _-( I' ) \mid I' \in \boldsymbol \tau, I \triangleleft I' \}$. Notice that such $I _*$ is unique since the minimal/maximal {entries} of segments of an output of the $\mathsf{ds} _m$-algorithm are all distinct. We have ${\boldsymbol \tau} = \varepsilon _{I_*,I} ( {\boldsymbol \tau}' )$ by inspection. By the minimality assumption on $I _*$, there exists no segment $I' \in {\boldsymbol \tau} - \{ I_*, I \}$ such that
$$I _* \cap I \subsetneq I' \subsetneq I _* \cup I.$$
By Theorem \ref{ADK}, we conclude that
$$\mathcal O _{{\boldsymbol \tau}'} \subset \overline{\mathcal O _{{\boldsymbol \tau}}} \text{ and } \dim \mathcal O _{{\boldsymbol \tau}} = \dim \mathcal O _{{\boldsymbol \tau}'} + 1.$$
We set $I = \{q_1 q ^{m_-}, \cdots, q_1 q ^{m_+} \}$. Then, we have
\begin{align}\nonumber
& \rho _{m_-,m_+ + 1} ( {\boldsymbol \tau}' ) = \rho _{m_-,m_+ + 1} ( {\boldsymbol \tau} ) - 1 \text{ and }\\
& \rho _{m_- - l,m_+ + 1} ( {\boldsymbol \tau}' ) = \rho _{m_- -l ,m_+ + 1} ( {\boldsymbol \tau} ) \text{ for every } l > 0. \label{rhoest}
\end{align}
Let $\tau _0 \in \mathcal{C} _{m _0} ( \sigma )$ be the marked partition obtained by setting $\tau _0 ^{\sharp}$ to be the collection of all hooks in $\sigma$ which gives a balanced segment along $m_0$. We have $\mathcal O_{{\boldsymbol \tau}_0} \subset \overline{\mathcal O _{{\boldsymbol \tau}}} \cap \overline{\mathcal O _{{\boldsymbol \tau}'}}$. Notice that each pair of segments of ${\boldsymbol \tau} _0 ^{\sharp}$ are nested to each other. Therefore, a repeated use of (\ref{rhoest}) claims that
\begin{equation}
\rho _{m_-,m_+ + 1} ( {\boldsymbol \tau}'' ) = \begin{cases} \rho _{m_-,m_+ + 1} ( {\boldsymbol \tau} ) - 1 & ( I \in ( {\boldsymbol \tau}'' ) ^{\sharp} ) \\ \rho _{m_-,m_+ + 1} ( {\boldsymbol \tau} )  & (I \not\in ( {\boldsymbol \tau}'' ) ^{\sharp})\end{cases}\label{rhoinc}
\end{equation}
for every $\tau'' \in  \mathcal{C} _{m _0} ( \sigma )$. Therefore, we have $\mathcal O _{\tau'} \subset \overline{\mathcal O _{\tau}}$ only if ${\boldsymbol \tau} ^{\sharp} \subset ( {\boldsymbol \tau}' ) ^{\sharp}$ by Theorem \ref{ADZ}.
 
\item Case {\bf n)}: The proof goes in a similar fashion if we replace ``$I \triangleleft I _*$" by ``$I _* \triangleleft I$", and $\min$ by $\max$.

\item This case-by-case analysis implies the result as desired.
\end{proof}

\begin{corollary}\label{smooth}
Keep the setting of Proposition \ref{cr}. Then $\overline{\mathcal O _{{\boldsymbol \tau}}}$ is smooth along $\mathcal O _{{\boldsymbol \tau}'}$.
\end{corollary}

\begin{proof}
We assume the setting of the proof of Proposition \ref{cr} and denote $s := \mathsf c _m ^\sigma$. The case {\bf 0)} is clear. We verify the assertion in the case {\bf p)}. Then, every relevant $G ( s )$-orbit is obtained as the pullback of an orbit of $\mathbf{Rep} ^{s}$ to $( \mathbf E ^{s} ( 0 ) \oplus \mathbf{Rep} ^{s} )$. By Theorem \ref{ADK} 2), we have an algebraic function $f _I$ on $( \mathbf E ^{s} ( 0 ) \oplus \mathbf{Rep} ^{s} )$ for each $I \in {\boldsymbol \tau} ^{\sharp} _0$ such that
$$\mathcal O_{{\boldsymbol \tau}} \subset \{ f _I = 0 \} \text{ if } I
\in {\boldsymbol \tau} ^{\sharp} \text{ and } \overline{\mathcal
  O_{{\boldsymbol \tau}} \cap \{ f _I \neq 0 \}} = \overline{\mathcal
  O_{{\boldsymbol \tau}}} \text{ if } I \not\in {\boldsymbol \tau}
^{\sharp}.$$ Moreover, we have $d f _I \not\equiv 0$ on $\mathcal O_{{\boldsymbol \tau}}$ by inspection. Therefore, $\{ f _I ; I \in ( {\boldsymbol \tau}' ) ^{\sharp} - {\boldsymbol \tau} ^{\sharp} \}$ gives an algebraically independent system of equation of $\mathcal O_{{\boldsymbol \tau}'}$ along an open dense subset of $\mathcal O _{{\boldsymbol \tau}}$. In particular, $\mathcal O_{{\boldsymbol \tau}'}$ is locally a complete intersection inside $\overline{\mathcal O _{{\boldsymbol \tau}}}$. This is the very definition of smoothness. Hence we have verified the case {\bf p)}. The case {\bf n)} is similar.\end{proof}

The proofs of the following three Theorems \ref{ctd}, \ref{multiplicity formula}, and \ref{minds} are simultaneously given in \S \ref{pmain}.

\begin{theorem}[Classification of tempered delimits]\label{ctd}
We have an equality $\mathcal C _{m_0} ( \sigma ) = \mathcal D _{m_0} ( \sigma )$, where $\mathcal C_{m_0}(\sigma)$ and $\mathcal D_{m_0}(\sigma)$ are as in (\ref{Csig}) and Definition \ref{d:tdl}, respectively.
\end{theorem}

\begin{theorem}\label{multiplicity formula}
Let $\tau \in \mathcal{D} _{m _0} ( \sigma )$ and let $\tau =
\tau ^{\sharp} \oplus \tau ^{\perp}$ be its standard decomposition as in Definition \ref{standard}. Fix $\tau _{\circleddash} \subset \tau ^{\sharp}$
and consider an induced
 decomposition $\tau = \tau _{\circleddash} \oplus \tau _{\circledast}$. We have
\begin{equation}
[ L ^{\mathsf A} _{{\boldsymbol \tau} _{\circleddash}} \boxast L _{{\boldsymbol \tau} _{\circledast}} ] = \sum _{\tau ' \in \mathcal{D} _{m_0} ( \sigma ); \tau _{\circledast} ^{\sharp} \subset ( \tau' ) ^{\sharp} \subset \tau _{\circleddash} \oplus \tau _{\circledast} ^{\sharp}} [ L _{{\boldsymbol \tau}'}] \in K ( \mathfrak{M} _{\vec{q}} ^n ),\label{inf}
\end{equation}
where $\tau' = ( \tau' ) ^{\sharp} \oplus ( \tau' ) ^{\perp}$ is the standard decomposition of $\tau' \in \mathcal{D} _{m_0} ( \sigma )$.
\end{theorem}

\begin{theorem}\label{minds}
Let $\tau _0 \in \mathcal D _{m_0} ( \sigma )$ be the minimal element with respect to the closure ordering $($i.e. $\tau _0 ^{\sharp}$ is maximal$)$. Then, $\lim _{m \to m_0} L _{{\boldsymbol \tau} ^{\perp} _0}$ is an irreducible discrete series.
\end{theorem}

\begin{corollary}[of Theorem \ref{multiplicity formula}]
Let $\tau \in \mathcal D _{m_0} ( \sigma )$ and $m_0 < m < m_0 + \frac{1}{2}$. We have
$$\tau ^{\mathtt e} _m \oplus \tau ^{\mathtt s} _m \in \mathcal D _{m_0} ( \sigma ).$$
In particular, $(\ref{inf})$ applied for $\tau _{\circleddash} := \tau ^{\mathtt e} _m$ and $\tau _{\circledast} := \tau ^{\mathtt s} _m$ for each $\tau \in \mathcal D _{m_0} ( \sigma )$ yields an overdetermined system of character equations.
\end{corollary}

{In section \ref{sec:formal}, we need the following Corollary
\ref{c:link} of Theorem \ref{multiplicity formula}.

\begin{definition}\label{d:linked}
We say that $\tau,\tau'$ in $\mathcal D_{m_0} ( \sigma )$ are linked if there
exist properly parabolically induced
  modules $V_1,\dots,V_k$ such that in the Grothendieck group of
  $\mathbb H_{n,m}$ we have:
$$[L_{\boldsymbol \tau}]+[L_{\boldsymbol \tau'}]\text{ or }[L_{\boldsymbol \tau}]-[L_{\boldsymbol \tau'}]\in \text{Span}_{\mathbb
  Z}([V_1],\dots,[V_k])\subset K(\mathfrak{M}^n_{\vec q}).$$ 
\end{definition}

\begin{corollary}\label{c:link}
Assume $m_0-\frac 12<m<m_0+\frac 12$ and $m \neq m_0$. Any two
tempered delimits in $\mathcal D_{m_0} ( \sigma )$ are linked (in the
sense of Definition \ref{d:linked}).
\end{corollary}

\begin{proof} We use induction on $h$, the number of balanced hooks at
  $m$, to show that there exists a system of  $2^h$ distinct equations (in the Grothendieck
  group) of the form
$$[L_{{\boldsymbol \tau}_i}]+[L_{{\boldsymbol \tau}_j}]=[V_{ij}],$$
where $V_{ij}$ is a properly parabolically induced modules, for $\tau_i, \tau_j \in \mathcal D_{m_0} ( \sigma ).$ Moreover, every $\tau \in \mathcal D_{m_0} ( \sigma )$ appears exactly $2^{h-1}$ times in these equations. Since
$2^{h-1}+2^{h-1}=2^h$, the claim follows.
\end{proof}

}

\subsection{Proofs of Theorems \ref{ctd}, \ref{multiplicity formula}, and \ref{minds}}\label{pmain}

We start with certain weaker versions of Theorems
\ref{minds} and \ref{multiplicity formula}, which turn out to
be sufficient in order to prove the full statements.

\begin{lemma}\label{Cminds}
Let $\tau _0 \in \mathcal C _{m_0} ( \sigma )$ be the minimal element with respect to the closure ordering $($i.e. $\tau _0 ^{\sharp}$ is maximal$)$. Then, $\lim _{m \to m_0} L _{{\boldsymbol \tau} ^{\perp} _0}$ must be an irreducible discrete series.
\end{lemma}

\begin{proof}
Let $\sigma'$ be a partition of $n'$ such that $\tau _0 ^{\perp} \in \mathcal C _{m_0} ( \sigma' )$. By the assumption $\tau _0 ^{\sharp}$ is maximal, we deduce that $\sigma'$ does not contains a balanced hook along $m_0$. 
Hence, we have
$$\tau ^{\perp} _0 = ds _{m'} ( \sigma' ) = ds _m ( \sigma' )$$
for $m_0 - \frac{1}{2} < m' < m_0 < m < m _0 + \frac{1}{2}$. In particular, $\lim _{m \to m_0} L _{{\boldsymbol \tau} _0 ^{\perp}}$ must be discrete series by Opdam-Solleveld \cite{OS}.
\end{proof}

\begin{proposition}\label{Cmf}
Let $\tau \in \mathcal{C} _{m _0} ( \sigma )$ and let $\tau = \tau ^{\sharp} \oplus \tau ^{\perp}$ be the standard decomposition. Fix $\tau _{\circleddash} \subset \tau ^{\sharp}$ and consider an induced decomposition $\tau = \tau _{\circleddash} \oplus \tau _{\circledast}$. We have
\begin{equation}
[ L ^{\mathsf A} _{{\boldsymbol \tau} _{\circleddash}} \boxast L _{{\boldsymbol \tau} _{\circledast}} ] = E + \sum _{\tau' \in \mathcal{D} _{m_0} ( \sigma ); \tau _{\circledast} ^{\sharp} \subset ( \tau' ) ^{\sharp} \subset \tau _{\circleddash} \oplus \tau _{\circledast} ^{\sharp}} [ L _{{\boldsymbol \tau}'}] \in K ( \mathfrak{M} _{\vec{q}} ^n ),\label{Ainf}
\end{equation}
where $\tau' = ( \tau' )^{\sharp} \oplus ( \tau' )^{\perp}$ is the standard decomposition, and $E$ is a non-negative sum of irreducible $\mathbb H _n$-modules which are not of the form $L _{\boldsymbol\tau}$ for any $\tau \in \mathcal{C} _{m _0} ( \sigma )$.
\end{proposition}

\begin{proof}
Notice that $L ^{\mathsf A} _{{\boldsymbol \tau} ^{\sharp}} = M ^{\mathsf A} _{{\boldsymbol \tau} ^{\sharp}}$. By Corollary \ref{indmult}, every composition factor $L _{{\boldsymbol \tau}'}$ of $L _{{\boldsymbol \tau} ^{\sharp}} ^{\mathsf A} \boxast L _{{\boldsymbol \tau} ^{\perp}}$ satisfies $\mathcal O _{{\boldsymbol \tau}} \subset \overline{\mathcal O _{{\boldsymbol \tau}'}}$. By Corollary \ref{smooth}, we deduce that $[ M _{{\boldsymbol \tau}} : L _{{\boldsymbol \tau}'} ] = 1$ for every $\tau' \in \mathcal C _{m_0} ( \sigma )$ such that $\mathcal O _{{\boldsymbol \tau}} \subset \overline{\mathcal O _{{\boldsymbol \tau}'}}$. Consequently, we have
$$
[ L _{{\boldsymbol \tau} ^{\sharp}} ^{\mathsf A} \boxast L _{{\boldsymbol \tau} ^{\perp}} : L _{{\boldsymbol \tau}'} ] \le 1 \text{ for every }\tau' \in \mathcal C _{m_0} ( \sigma ).$$
By induction-by-stages, we have
\begin{equation}
L ^{\mathsf A} _{{\boldsymbol \tau} ^{\sharp}} \boxast M _{{\boldsymbol \tau} ^{\perp}} \cong L _{I _1} ^{\mathsf A} \boxast ( L _{I _2} ^{\mathsf A} \boxast ( \cdots \boxast ( L ^{\mathsf A} _{I_N} \boxtimes M _{{\boldsymbol \tau} ^{\perp}} ))),\label{ibs}
\end{equation}
where ${\boldsymbol \tau} ^{\sharp} = \{ I _k \} _{k=1} ^N$. If $N =
0$, there is nothing to prove. Let $n' := n - \# I _1$ and $\sigma'$ be the Young diagram obtained by extracting a hook corresponding to $I_1$ from $\sigma$. Consider the following assertion:
\begin{itemize}
\item[$(\heartsuit)$] The $\mathbb H$-module $L _{I_1} ^{\mathsf A} \boxast L _{\widetilde{\boldsymbol \tau}}$ contains both of $L _{\widetilde{\boldsymbol \tau} ^{(1)}}$ and $L _{\widetilde{\boldsymbol \tau} ^{(2)}}$ as composition factors for every $\widetilde{\tau} \in \mathcal C _{m_0} ( \sigma' )$, and $\widetilde{\tau} ^{(1)}, \widetilde{\tau} ^{(2)} \in \mathcal C _{m_0} ( \sigma )$ that satisfy $\widetilde{\boldsymbol \tau} ^{\sharp} = ( \widetilde{\boldsymbol \tau} ^{(1)} ) ^{\sharp}$ and $\widetilde{\boldsymbol \tau} ^{\sharp} \cup \{ I_1 \} = ( \widetilde{\boldsymbol \tau} ^{(2)} ) ^{\sharp}$.
\end{itemize}
If $(\heartsuit)$ holds, and (\ref{Ainf}) holds for all smaller $N$, then the comparison of multiplicity yields $[ L _{I_1} ^{\mathsf A} \boxast L _{\widetilde{\boldsymbol \tau}} ] = E' + [L _{\widetilde{\boldsymbol \tau} ^{(1)}}] + [L _{\widetilde{\boldsymbol \tau} ^{(2)}}]$, where $E'$ is a non-negative (formal) linear combination of irreducible $\mathbb H _n$-modules which are not isomorphic to $L _{( \mathsf{c} _{\sigma} ^m, \tau )}$ for some $\tau \in \mathcal{C} _{m _0} ( \sigma )$. Therefore, in order to prove (\ref{ibs}), it suffices to verify $(\heartsuit)$ provided that (\ref{Ainf}) holds for all smaller $N$ cases.

Set $\widetilde{\boldsymbol \tau} ^+ := \{ I _1 \} \oplus \widetilde{\boldsymbol \tau}$. We have $\widetilde{\boldsymbol \tau} ^+ \in \mathcal C _{m_0} ( \sigma )$. By Corollary \ref{indmult}, every composition factor $L _{{\boldsymbol \tau}'}$ of $( L _{I_1} ^{\mathsf A} \boxast L _{\widetilde{\boldsymbol \tau}} )$ satisfies $\mathcal O _{\widetilde{\boldsymbol \tau} ^+} \subset \overline{\mathcal O _{{\boldsymbol \tau}'}}$. By Corollary \ref{smooth}, we deduce that $[ M _{\widetilde{\boldsymbol \tau} ^+} : L _{{\boldsymbol \tau}'} ] = 1$ for every $\tau' \in \mathcal C _{m_0} ( \sigma )$ such that $\mathcal O _{\widetilde{\boldsymbol \tau} ^+} \subset \overline{\mathcal O _{{\boldsymbol \tau}'}}$. To show $(\heartsuit)$, it suffices to verify that
$$[ L _{I_1} ^{\mathsf A} \boxast L _{\widetilde{\boldsymbol \tau}'} : L _{\widetilde{\boldsymbol \tau} ^{(i)}} ] = 0$$
for $i=1,2$ and every irreducible constituent {$L _{\widetilde{\boldsymbol \tau}'} \not\cong L _{\widetilde{\boldsymbol \tau}}$ of $M _{\widetilde{\boldsymbol \tau}}$.}
This follows if
\begin{itemize}
\item[$(\heartsuit)'$] $\overline{\mathcal O _{\widetilde{\boldsymbol \tau} ^{(i)}}}$ does not contain $\mathcal O _{ \{ I _1 \} \oplus \widetilde{\boldsymbol \tau}'}$
\end{itemize}
holds by Corollary \ref{indmult}. Here we have
\begin{align}\nonumber
& \dim \mathcal O _{\{ I _1 \} \oplus \widetilde{\boldsymbol \tau}'} > \dim \mathcal O _{\widetilde{\boldsymbol \tau} ^+} \text{ and }\\
& \dim \mathcal O _{\widetilde{\boldsymbol \tau} ^{(i)}} \le \dim \mathcal O _{\widetilde{\boldsymbol \tau} ^+} + 1 \text{ (for each $i=1,2$)}.
\end{align}
Thus, in order to deduce inclusion, we have necessarily $\dim \mathcal
O _{\{ I _1 \} \oplus \widetilde{\boldsymbol \tau}'} = \dim \mathcal O
_{\widetilde{\boldsymbol \tau} ^+} + 1$. It follows that $\widetilde{\boldsymbol \tau}'$ is
obtained from $\widetilde{\boldsymbol \tau}$ by applying a unique elementary
modification or putting one extra marking. Here $I _1$ corresponds to
a hook of $\sigma$, but $I_1$ does not correspond to a hook of $\sigma'$. Therefore, we conclude that $\{ I _1 \} \oplus
\widetilde{\boldsymbol \tau}' \neq \widetilde{\boldsymbol \tau} ^{(i)}$ for $i=1,2$. This in
turn implies $(\heartsuit)'$, and hence $(\heartsuit)$. In conclusion,
the induction proceeds and we obtain the result.
\end{proof}

The rest of this section is devoted to the proof of Theorems \ref{ctd}, \ref{multiplicity formula}, and \ref{minds}.

We apply Proposition \ref{Cmf} to $\tau _0$ (borrowed from Lemma \ref{Cminds}). Then, we obtain
$$[ L _{{\boldsymbol \tau} _0 ^{\sharp}} ^{\mathsf A} \boxast L _{{\boldsymbol \tau} _0 ^{\perp}} ] = E + \sum _{\tau \in \mathcal{C} _{m_0} ( \sigma )} [ L _{{\boldsymbol \tau}}] \in K ( \mathfrak{M} _{\vec{q}} ^n )$$
by Proposition \ref{cr}. Here, $\lim _{m \to m_0} L _{{\boldsymbol \tau} _0 ^{\sharp}} ^{\mathsf A}$ is a tempered module while $\lim _{m \to m_0} L _{{\boldsymbol \tau} _0 ^{\perp}}$ is a well-defined discrete series. It follows that every irreducible constituent of $L _{{\boldsymbol \tau} _0 ^{\sharp}} ^{\mathsf A} \boxast L _{{\boldsymbol \tau} _0 ^{\perp}}$ is a tempered delimit. In particular, we have $\mathcal C _{m_0} ( \sigma ) \subset \mathcal D _{m_0} ( \sigma )$. Moreover, Corollary \ref{DinC} implies that $\mathcal C _{m_0} ( \sigma ) = \mathcal D _{m_0} ( \sigma )$ by the comparison of the cardinality. This proves Theorem \ref{ctd} and hence also Theorem \ref{minds}. Moreover, we conclude {$E = 0$} since there can be no other tempered delimits outside of $\mathcal C _{m_0} ( \sigma )$. Therefore, we conclude Theorem \ref{multiplicity formula} as desired.

\subsection{Further properties of tempered delimits}

We first recall a result based on the theory of analytic $R$-groups due, in the setting of affine Hecke algebras, to Delorme-Opdam \cite{DO}:

\begin{theorem}[Slooten \cite{Sl2} Theorem 3.4.4]\label{Rgroup}
Let $\sigma$ be a partition of $n$. Let $\mathbf I = \{ I_1, \ldots, I
_N \}$ be a multisegment consisting of segments with $\mathtt E _{m} (
I _k ) = 1$. Let $d$ be the number of segments $I_k$ of
  distinct size such that $e _+ ( I _k ) \not\in \{ e _+ ( I )
\mid I \in \mathsf{ds} _m ( \sigma) \}$ and $e _- ( I _k ) \not\in \{
e _- ( I ) \mid I \in \mathsf{ds} _m ( \sigma) \}$. Then, the module $L ^{\mathsf A} _{\mathbf I} \boxast \mathsf{ds} _m ( \sigma )$ is irreducible when $m \neq m_0$, and is a direct sum of $2^d$ irreducible components when $m = m _0$.
\end{theorem}

\begin{corollary}\label{irred}
For every $\tau \in \mathcal D _{m_0} ( \sigma )$, the limit module $\lim _{m \to m_0} L _{\boldsymbol \tau}$ is irreducible. 
\end{corollary}

\begin{proof}
We borrow the notation $\boldsymbol \tau_0$ from Theorem \ref{minds}. By Theorem \ref{Rgroup}, $\lim _{m \to m_0} L ^{\mathsf A} _{{\boldsymbol \tau} _0 ^{\sharp}} \boxast L _{{\boldsymbol \tau} _0 ^{\perp}}$ splits into $2^h$ direct sums of tempered modules, where $h$ is the number of segments in ${\boldsymbol \tau} _0 ^{\sharp}$. By Theorem \ref{multiplicity formula}, we know that $L ^{\mathsf A} _{{\boldsymbol \tau} _0 ^{\sharp}} \boxast L _{{\boldsymbol \tau}  _0^{\perp}}$ contains $2 ^h$ irreducible constituent even at generic $m$. It follows that all of such irreducible constituents, which is the whole of $\mathcal D _{m_0} ( \sigma )$, must be irreducible by taking limit $m \to m_0$.
\end{proof}

\begin{lemma}\label{disj}
For every distinct choice of partitions $\sigma, \sigma'$ of $n$, we have
$$\mathcal D _{m_0} ( \sigma ) \cap \mathcal D _{m_0} ( \sigma' ) = \emptyset.$$
\end{lemma}

\begin{proof}
Let $\chi\in \mathsf P(\vec{q})$ be a parameter corresponding to an
element of $\mathcal D _{m_0} ( \sigma )$. It is sufficient to prove that
$\sigma$ is canonically recovered from $\chi _0 := \lim _{m \to m_0}
\chi^0\in \mathsf Q(q)$.

 Let $I \in \chi_0$ be a segment so that {\bf 1)} $I$ or
$I^{-1}$ is of the form $\{ q^{m_0}, q^{m_0+1}, \ldots \}$ or
$\{\ldots, q^{m_0-1},q^{m_0}\}$, and {\bf 2)} $\max (I \cup I^{-1}
)$ attains the maximum among all the segments in $\chi_0$
which satisfy condition {\bf 1)}. Such a segment $I$ must be
unique (if it exist) since it gives the first segment in the smallest nested component with respect to $\Subset$ (via the $\mathsf{ds} _m$-algorithm, see also Definition \ref{Cdef}).

Let $\chi' = ( \mathbf I', \delta )$ be the marked subpartition of $\chi$ so that $\mathbf I'$ is the collection of all segments $I'$ such that $I \Subset I'$. By the $\mathsf{ds}_m$-algorithm and Definition \ref{Cdef}, we deduce that $\chi'$ forms a nested component of $\chi$ such that either both $\min I'$ and $\max I'$ or both $( \min I' )^{-1}$ and $( \max I' )^{-1}$ are the maximal/minimal values of a hook extracted from $\sigma$. Moreover, the marking of $I'$ determines either $I'$ or $(I')^{-1}$ must belong to $\chi$, and consequently we obtain the shape of all the intermediate segments of the $\mathsf{ds} _m$-algorithm step 2). Therefore, $\chi'$ is determined uniquely from $\chi _0$.

In particular, we can assume $\chi' = \emptyset$. Then, according to the marking of $I$, all the segments of $\chi$ must be either uniformly marked (after changing the marking within the equivalence class if necessary) or uniformly unmarked. It follows that every $I' \in \chi$ must satisfy $\mathtt E_m (I') > 1$ or $\mathtt E_m (I') < 1$ uniformly. Assume that $\mathtt E_m (I') > 1$ for every $I' \in \chi$. Then, we arrange $\chi$ as
$$\max I = \max I_1 > \max I_2 > \cdots > \max I _\ell.$$
A segment $I_k$ appears in the $\mathsf{ds} _m$-algorithm step 2) if and only if $I = I_1$ or $\min I_k = q^{-1} \min I_l$ for some $l < k$. All the others are union of two segments appearing in the $\mathsf{ds} _m$-algorithm step 2). This recovers the all segments appearing in the $\mathsf{ds}_m$-algorithm step 2), and hence recovers $\sigma$ uniquely. The other case is completely analogous, and hence the result follows.
\end{proof}

\begin{theorem}\label{dind}
Every tempered irreducible $\mathbb H _{n,m_0}$-module is obtained as
$$\lim _{m \to m_0} L ^{\mathsf A} _{\mathbf I} \boxast L _{\boldsymbol \tau}$$
for a unique balanced multisegment $\mathbf I$ $($along $m_0)$ and $\tau \in \mathcal D _{m_0} ( \sigma )$ for a unique partition $\sigma$ of some $n'$. In particular, such $\lim _{m \to m_0} L ^{\mathsf A} _{\mathbf I} \boxast L _{\boldsymbol \tau}$ is an irreducible $\mathbb H _{n,m_0}$-module.
\end{theorem}

\begin{proof}
By the Evens-Langlands classification \cite{Ev} (see also \cite{CK}), a tempered module is written as a quotient of a parabolic induction of the form $L ^{\mathsf A} _{\mathbf I'} \boxast \mathsf{ds}$, where $\mathbf I'$ is a multisegment and $\mathsf{ds} = \lim _{m \to m_0} \mathsf{ds} _m ( \sigma' )$ is a discrete series obtained from a partition $\sigma'$. Since every discrete series is a tempered delimit, we can further assume that $\sigma'$ does not contain a balanced hook as in Theorem \ref{minds}. (c.f. Theorem \ref{multiplicity formula})

We set $\mathbf I''$ to be the collection of all distinct segments of $\mathbf I'$ so that we have $e _+ ( I'' ) \not\in \{ e _+ ( I ) \mid I \in \mathsf{ds} _m ( \sigma' ) \}$ and $e _- ( I'' ) \not\in \{ e _- ( I ) \mid I \in \mathsf{ds} _m ( \sigma' ) \}$ for every $I'' \in \mathbf I''$. Let $\mathbf I$ be the multisegment obtained from $\mathbf I'$ by removing segments in $\mathbf I''$. By the irreducibility of tempered induction of affine Hecke algebras of $\mathop{GL}(n)$, we deduce $L ^{\mathsf A} _{\mathbf I} \boxast ( L ^{\mathsf A} _{\mathbf I''} \boxast \mathsf{ds} ) \cong L ^{\mathsf A} _{\mathbf I'} \boxast \mathsf{ds}$. Then, Theorem \ref{Rgroup} implies that both $L ^{\mathsf A} _{\mathbf I''} \boxast \mathsf{ds}$ and $L ^{\mathsf A} _{\mathbf I'} \boxast \mathsf{ds}$ share the same number of irreducible constituents. It follows that $L ^{\mathsf A} _{\mathbf I} \boxast L$ is irreducible for every irreducible constituent $L$ of $L ^{\mathsf A} _{\mathbf I'} \boxast \mathsf{ds}$. Here $\mathbf I' \oplus \mathsf{mp} _m ( \sigma' )$ is adapted to $\mathsf{c} _m ^{\sigma}$ for the larger partition $\sigma$ of $n'$ by the construction of $\mathbf I'$. Therefore, we conclude $L = L _{\boldsymbol \tau}$ for some $\tau \in \mathcal D _{m_0} ( \sigma )$, which implies the existence part of the assertion.

We prove the uniqueness of $\sigma$ and $\tau \in \mathcal D _{m_0} ( \sigma )$. Since the set $\mathbf I$ is uniquely determined by $\chi \in \mathsf{P} (\vec{q})$ corresponding to the tempered module, we can assume $\mathbf I = \emptyset$. Then, the assertion reduces to Lemma \ref{disj} as desired.
\end{proof}

\begin{corollary}[of Corollary \ref{smooth} and Theorem \ref{multiplicity formula}]\label{diff}
Assume that some $\tau _{\max} \in \mathcal D _{m_0} ( \sigma )$ satisfies $\mathsf{sgn} \subset L _{{\boldsymbol \tau} _{\max}}$ $($as $W$-modules$)$. Then, for every $\tau \in \mathcal D _{m_0} ( \sigma )$, we have
\begin{equation}
M _{\boldsymbol \tau} \cong L _{\boldsymbol \tau ^{\sharp}} \boxast L _{\boldsymbol \tau ^{\perp}}.\label{std-td}
\end{equation}
In addition, if $m_0 = 1/2$ or $1$, then for each such $\tau$, there exist
\begin{itemize}
\item a nilpotent element $x = x _{\tau} \in \mathfrak g$ $($with $\mathfrak g = \mathfrak{so}_{2n+1}$ and $\mathfrak{sp}_{2n}$, respectively$)$;
\item an irreducible representation $\xi = \xi _{\tau}$ of $A_x$ appearing in the Springer correspondence, where  $A _x = \mathsf{Stab} _G ( x _{\tau }) / \mathsf{Stab} _G ( x _{\tau }) ^{\circ}$ $($with $G = \mathop{SO} (2n+1)$ and $\mathop{Sp} ( 2n )$, respectively$)$;
\end{itemize}
such that we have
\begin{equation}\label{ident}
[ H _{\bullet} ( \mathcal B _x ) _{\xi} ] = \sum _{\tau' \in \mathcal D _{m_0} ( \sigma ) ; ( \tau' )^{\sharp} \subset \tau ^{\sharp}} (-1) ^{\# {\boldsymbol \tau}^{\sharp} - \# ( {\boldsymbol \tau}' ) ^{\sharp}} [ H _{\bullet} ( \mathcal E _{\boldsymbol \tau'} ) ]
\end{equation}
as virtual $W$-modules $($without gradings$)$. Here $\mathcal B _x$ is the Springer fiber of $($the flag variety of $G$ along$)$ $x$ and the subscript $\xi$ means the $\xi$-isotypic component {as $A_x$-modules}.
\end{corollary}

\begin{proof}
Choose $m_0 < m < m_0 + \frac{1}{2}$. By \cite{CK} Theorem 1.22, our assumption is equivalent to the fact that $\mathcal O _{{\boldsymbol \tau}_{\max}}$ is the open dense orbit of $\mathbb V ^{(\mathsf{c} ^{\sigma}_m,\vec{q})}$. Taking into account Theorem \ref{multiplicity formula}, Corollary \ref{smooth}, {and the Ginzburg theory (c.f. \cite{K1} Theorem 11.2),} it suffices to prove that $\mathcal O _{{\boldsymbol \tau}}$ is not contained in any orbit closure except for $\mathcal O _{{\boldsymbol \tau}'}$ for some $\tau' \in \mathcal D _{m_0} ( \sigma )$ in order to prove (\ref{std-td}).

We prove this by contradiction. Let $\tau \in \mathcal D _{m_0} ( \sigma )$ be a marked partition so that there exists a parameter $\chi$ which does not come from $\mathcal D _{m_0} ( \sigma )$, but $\mathcal O _{\boldsymbol \tau} \subset \overline{\mathcal O _{\chi}}$. If $\tau ^- \neq \emptyset$, then we deduce that $\tau _{\max}$ cannot define an open dense orbit. Thus, we assume $\tau ^- = \emptyset$ in the following. Let $I_1, I_2, \ldots$ be the set of segments obtained from balanced hooks of $\sigma$ along $m_0$. For each $k$, there exists at most one segment $I$ of $\boldsymbol \tau$ such that $I _k \triangleleft I$; this is because $\mathcal O _{{\boldsymbol \tau}_{\max}}$ is open dense. Thus, we further conclude that $\varepsilon_{I_k,I} ( {\boldsymbol \tau} )$ comes from $\mathcal D _{m_0} ( \sigma )$ for every $I \in {\boldsymbol \tau}$. Therefore, $\chi$ cannot exist, which finishes the first part of the proof.

For the latter part, notice that every tempered $\mathbb H _{n,m_0}$-module is of the form $H _{\bullet} ( \mathcal B _x ) _{\xi}$ by Kazhdan-Lusztig \cite{KL2} Theorem 8.2 (before taking fixed points, but this does not affect the $\mathbb H _{n,m_0}^f$- or $W$-module structures). Therefore, inverting the multiplicity matrix (of simple modules in standard modules as $\mathbb H _{n,m}$-modules) given by (\ref{std-td}) and Theorem \ref{multiplicity formula} for $\mathcal D _{m_0} ( \sigma )$ yields the assertion for $m$. By Corollary \ref{irred}, the assertion holds for the limit $m \to m _0$ as required.
\end{proof}

\begin{remark} 
\begin{enumerate}
\item[(a)] If $\xi = 1$, then the result becomes simpler as in \cite{CK} Corollary 1.23 (the right hand side of (\ref{ident}) is only $[H_\bullet(\mathcal E_{\tau_{\max}})]$).
\item[(b)]In view of Lusztig \cite{L5} Theorem 1.22, the same
  statement holds at all critical values $m_0\ge 3/2$, but the
  homology of the classical Springer fiber in the left hand side of
  (\ref{ident}) is replaced by the appropriate homology group. It
  should be added that Corollary \ref{diff} becomes weaker for larger
  values of $m_0$, since it becomes more difficult for $\tau_{\max}$ with the desired property to exist.
\end{enumerate}
\end{remark}

\subsection{An inductive algorithm for characters of tempered delimits}\label{sec:Walg}

We give an inductive algorithm for computing $W$-characters of
  tempered delimits, and in particular of discrete series and limits
  of discrete series for all values of the parameter $m$. Fix a partition $\sigma$ of $n$, and we retain the notation as before. In the
  following
$$\Theta_W(\pi)=\sum_{\chi\in\widehat W} [\pi : \chi]_{\mathbb C [W]} \, \chi$$
denotes the $W$-character of a module $\pi.$ The induction proceeds in two ways: increasingly on the rank $n$, and decreasingly on the value $m$ of the parameter.

\begin{remark}
Taking into account Theorem \ref{dind}, we know that the $W$-character of every tempered module is obtained as an induced module of a suitable tempered delimits.
\end{remark}

First, we fix a parameterization of $\widehat W_n$: Let $\mathbb C [ \epsilon_1, \ldots, \epsilon_n ]$ be the polynomial ring in which $W _n$ acts naturally by extending the action on $R \subset X ^* ( T )$.
For a bi-partition $( \mu, \nu )$ of $n$, we define
\begin{align*}
\mathsf{p} ^0 _i ( \mu ) := \prod _{\mu _i ^< < k < l \le \mu _i ^{\le}} ( \epsilon _k ^2 - \epsilon _l ^2 ), & \ \mathsf{p} ^+ _i ( \mu, \nu ) := \prod _{\nu _i ^< < k < l \le \nu _i ^{\le}} ( \epsilon _{k + \left| \mu \right|} ^2 - \epsilon _{l + \left| \mu \right|} ^2 ), \text{ and }\\
\mathsf{p} ( \mu, \nu ) := & \prod _{i > \left| \mu \right|} \epsilon _i \times \prod _{i \ge 1} \mathsf{p} ^0 _i ( \mu ) \mathsf{p} ^+ _i ( \mu, \nu ).
\end{align*}

\begin{lemma}[cf. \cite{Ca},\S 11.4]\label{Wtypes}
Let $( \mu, \nu )$ be a bi-partition of $n$.
Then the following $W _n$-module is irreducible:
$$\{\mu, \nu \} := \mathbb C [ W _n ] \mathsf p ( {} ^{\mathtt t} \mu, {}^{\mathtt t} \nu ) \subset \mathbb C [  \epsilon_1, \ldots, \epsilon_n ].$$
\end{lemma}

\smallskip

\begin{algorithm}\label{algW}
\item {\bf Step 0.} Let $m_0=n-1$ be the critical point. If $m>n-1$ then \cite{CK} gives that
$\Theta_W(\mathsf{ds}_m(\sigma))= \{\emptyset,{} ^{\mathtt t}\sigma \}.$

\item {\bf Step 1.} Consider
$m_0<m<m_0+\frac 12.$ Applying Theorem \ref{multiplicity formula} in $\mathcal
D_{m_0}(\sigma)$, we find that 
\begin{equation}
\Theta_W( L ^{\mathsf A} _{{\boldsymbol \tau} _m^{\mathtt e}} \boxast L _{{\boldsymbol \tau} _{m}^{\mathtt s}}
 ) = \sum _{\tau' \in \mathcal{D} _{m_0} ( \sigma ); ( {\boldsymbol \tau}' )^{\mathtt e}
  _{m}  \subset {\boldsymbol \tau} _m^{\mathtt e}} \Theta_W(L _{{\boldsymbol \tau}'_m}).  
\end{equation}
Solving this linear system gives
\begin{equation}
\Theta_W(L_{\boldsymbol \tau})=\sum_{\tau' \in \mathcal{D} _{m _0} ( \sigma ); ( {\boldsymbol \tau}' )^{\mathtt e}
  _{m}  \subset {\boldsymbol \tau} _m^{\mathtt e}} (-1)^{\#({\boldsymbol \tau}_m^{\mathtt e} \setminus ( {\boldsymbol \tau}' )^{\mathtt e}_m)}~
\Theta_W( L ^{\mathsf A} _{( {\boldsymbol \tau}' )^{\mathtt e} _m} \boxast L _{( {\boldsymbol \tau}' )^{\mathtt s}
  _{m}} ),\quad \tau \in \mathcal D_{m_0} ( \sigma )\label{transf}
\end{equation}
Notice that in this equation, the right hand side is known since $( {\boldsymbol \tau}' )^{\mathtt s} = \mathsf{ds} _m ( \sigma' )$ for some partition $\sigma'$ of $n' \le n$. 


\item {\bf Step 2.} Set $m=m_0$. By Corollary \ref{irred}, for every $\tau \in \mathcal
D_{m_0}(\sigma)$, the module $L ( \tau ) := \lim_{m\to
  m_0} L_{\boldsymbol \tau}$ is irreducible. Hence by Corollary \ref{irred} and by previous step, we get the $W$-character
$\Theta_W( L ( \tau ) )$ for every $\tau\in \mathcal D_{m_0}( \sigma ).$


\item {\bf Step 3.} Consider $m_0-\frac 12<m<m_0.$ We need to find
$\Theta_W(\mathsf{ds}_m(\sigma))$ (in order proceed with the
algorithm). By Corollary \ref{irred}, there exists a unique $\tau\in \mathcal D_{m_0}(\sigma),$ such that $\lim_{m\to m_0} \mathsf{ds}_m(\sigma)=L ( \tau ),$ and in particular,
$\Theta_W(\mathsf{ds}_m(\sigma))=\Theta_W(L ( \tau ) ).$ Applying
       Lemma \ref{Cbasic}, this $\tau$ is characterized by
       ${\boldsymbol \tau}^\sharp,$ which is precisely the set of
       unmarked balanced segments of $\mathsf{ds}_m(\sigma)$.


\item {\bf Step 4.} If $m_0>1-n$, set $m_0=m_0-\frac 12$ and
move to {\bf Step 1}.
\end{algorithm}

\begin{remark}\label{RTch}
Notice that (\ref{transf}) remains valid if we replace $\Theta_W$ with $\mathsf{ch}$. Moreover, we can compute $\mathsf{ch} ( \mathsf{ds}_m (\sigma) )$ for $m \gg 0$ (by \cite{CK} \S 4.7), and $\mathsf{ch} ( L^{\mathsf A} \boxast L )$ from $\mathsf{ch} L^{\mathsf A}$ and $\mathsf{ch} L$. Since $\mathsf{ch} L_{\boldsymbol \tau}$ depends on $m$ holomorphically in the region $m_0 - \frac{1}{2} < m < m_0 + \frac{1}{2}$ for each $\tau \in \mathcal D_{m_0} (\sigma)$, we can also compute $\mathsf{ch} L_{\boldsymbol \tau}$ by using the above algorithm.
\end{remark}

\smallskip

\begin{example} Consider the case $n=6$ and $\sigma=(2,2,2)$, so that
  we have $\{\epsilon _i ( \mathsf{c}_m^\sigma) \}_{i=1}^6 = \{q ^{m+1},q^m,q^m,q^{m-1},q^{m-1},q^{m-2}\}$. 
Then we find the
  following cases:

\begin{enumerate}

\item $m_0<m<m_0+\frac 12$ with $m_0\in \frac 12\mathbb Z$ and $m_0\ge 2$. We have $\mathcal
  D_{m_0}(\sigma)=\{\mathsf{ds}_m(\sigma)\},$ where
  $\Theta_W(\mathsf{ds}_m(\sigma))=\{(0)(3^2)\}.$

\item $\frac 32<m<2$. We have
  $\mathcal D_{\frac 32} (\sigma)=\{\mathsf{ds}_m(\sigma),L_{\tau^1_m}\}$, where:
  \begin{enumerate}
  \item $\Theta_W(\mathsf{ds}_m(\sigma))=\{(0)(3^2)\}$; 
  \item
    $\Theta_W(L_{\tau_m^1})=\{(0)(321)\}+\{(0)(2^21^2)\}+\{(1)(2^21)\}+\{(1)(32)\}+\{(1^2)(2^2)\}$
    ($({\boldsymbol \tau}_m^1)^{\mathtt e}=\{ q^{m-2}, q^{m-1}
    \}$).
\end{enumerate}

\item $m=\frac 32$. $\mathcal D_{\frac 32}(\sigma)$ is as before, but $\lim_{m\searrow \frac
  32}\mathsf{ds}_m(\sigma)$ is not a discrete series.

\item $1<m<\frac 32$. We have $\mathcal
  D_{1} ( \sigma) =\{\mathsf{ds}_m(\sigma),L_{{\tau}_m^1},L_{{\tau}_m^2},L_{{\tau}_m^3}\}$, where:
\begin{enumerate}
  \item $\Theta_W(\mathsf{ds}_m(\sigma))=\{(0)(321)\}+\{(0)(2^21^2)\}+\{(1)(2^21)\}+\{(1)(32)\}+\{(1^2)(2^2)\};$  
  \item
    $\Theta_W(L_{\tau_m^1})=\{(0)(31^3)\}+\{(0)(21^4)\}+\{(1)(31^2)\}+\{(1)(21^3)\}+\{(1^2)(31)\}+\{(1^2)(21^2)\}+\{(1^3)(21)\}$
    ($({\boldsymbol \tau}_m^1)^{\mathtt e}=\{ q^{m-2}, q^{m-1} \}$);
  \item
    $\Theta_W(L_{\tau_m^2})=\{(1)(2^21)\}+\{(2)(2^2)\}+\{(0)(2^3)\}$  ($({\boldsymbol \tau}_m^2)^{\mathtt e}= \{ q^{m-1} \}$);
  \item
    $\Theta_W(L_{\tau_m^3})=\{(0)(1^6)\}+\{(0)(2,1^4)\}+\{(0)(2^21^2)\}+2\{(1)(1^5)\}+2\{(1^2)(1^4)\}+2\{(1^3)(1^3)\}+\{(1^4)(1^2)\}+2\{(1)(21^3)\}+\{(2)(1^4)\}+2\{(1^2)(21^2)\}+\{(2)(21^2)\}+\{(1)(2^21)\}+\{(21)(1^3)\}+\{(1^3)(21)\}+\{(21^2)(1^2)\}+\{(1^2)(2^2)\}+\{(21)(21)\}$   ($({\boldsymbol \tau}_m^3)^{\mathtt e}=\{ ({\boldsymbol \tau}_m^1)^{\mathtt e}, ({\boldsymbol \tau}_m^2)^{\mathtt e}\}$).
\end{enumerate}

\item $m=1$. $\mathcal D_{1}(\sigma)$ is as before, but $\lim_{m\searrow 1}\mathsf{ds}_m(\sigma)$ is not a discrete series.

\item $\frac 12<m<1$. We have $\mathcal
  D_{\frac 12} (\sigma)=\{\mathsf{ds}_m(\sigma),L_{{\tau}_m^1},L_{{\tau}_m^2},L_{{\tau}_m^3}\}$, where:
\begin{enumerate}
  \item
    $\Theta_W(\mathsf{ds}_m(\sigma))=\{(0)(1^6)\}+\{(0)(21^4)\}+\{(0)(2^21^2)\}+2\{(1)(1^5)\}+2\{(1^2)(1^4)\}+2\{(1^3)(1^3)\}+\{(1^4)(1^2)\}+2\{(1)(21^3)\}+\{(2)(1^4)\}+2\{(1^2)(21^2)\}+\{(2)(21^2)\}+\{(1)(2^21)\}+\{(21)(1^3)\}+\{(1^3)(21)\}+\{(21^2)(1^2)\}+\{(1^2)(2^2)\}+\{(21)(21)\};$ 
  \item
    $\Theta_W(L_{\tau_m^1})=\{(1^3)(21)\}+\{(21^2)(2)\}+\{(1^4)(2)\}+\{(1^4)(1,1)\}+\{(21^3)(1)\}+\{(1^5)(1)\}$
    ($({\boldsymbol \tau}_m^1)^{\mathtt e}= \{q^{m+1},q^m,q^{m-2},q^{m-2} \}$);
  \item
    $\Theta_W(L_{\tau_m^2})=\{(1^2)(1^4)\}+\{(21)(1^3)\}+\{(2^2)(1^2)\}$  ($({\boldsymbol \tau}_m^2)^{\mathtt e}=\{ q^m,q^{m-1} \}$);
  \item
    $\Theta_W(L_{\tau_m^3})=\{(1^3)(1^3)\}+\{(1^4)(1^2)\}+\{(21^2)(1^2)\}+\{(21^3)(1)\}+\{(1^5)(1)\}+\{(1^6)(0)\}+\{(2,1^4)(0)\}+\{(2^21^2)(0)\}+\{(2^21)(1)\}$
    ($({\boldsymbol \tau}_m^3)^{\mathtt e}=\{ ({\boldsymbol \tau}_m^1)^{\mathtt e}, ({\boldsymbol \tau}_m^2)^{\mathtt e}\}$).
\end{enumerate}

\item $m=\frac 12$. $\mathcal D_{\frac 12}(\sigma)$ is as before, but $\lim_{m\searrow \frac 12}\mathsf{ds}_m(\sigma)$ is not a discrete series.

\item $0<m<\frac 12$. We have $\mathcal
  D_{0}(\sigma)=\{\mathsf{ds}_m(\sigma),L_{{\tau}_m^1}\}$, where:
\begin{enumerate}
  \item $\Theta_W(\mathsf{ds}_m(\sigma))=\{(1^3)(1^3)\}+\{(1^4)(1^2)\}+\{(21^2)(1^2)\}+\{(21^3)(1)\}+\{(1^5)(1)\}+\{(1^6)(0)\}+\{(21^4)(0)\}+ \{(2^21^2)(0)\}+\{(2^21)(1)\}$;
  \item
    $\Theta_W(L_{\tau_m^1})=\{(2^3)(0)\}$
    ($({\boldsymbol \tau}_m^1)^{\mathtt e}= \{q^{m+1},q^m,q^{m-1}\}$).
\end{enumerate}

\item $m=0$. $\mathcal D_0(\sigma)$ is as before, but $ \lim_{m\searrow 0}\mathsf{ds}_m(\sigma)$ is not a discrete series.

\item $m_0<m<m_0 + \frac{1}{2}$ with $m_0 \in \frac{1}{2} \mathbb Z _{< 0}$. We have $\mathcal D_{m_0}(\sigma)=\{\ds_m(\sigma)\}$, where $\Theta_W(\ds_m(\sigma))=\{(2^3)(0)\}.$

\end{enumerate}

\qed
\end{example}

We explain that the algorithm gives also the $W^{\mathsf
    D}_n$-character of discrete series for $\mathbb H_{n,m}^{\mathsf
    D}$. For every partition $\sigma$ of $n$, we set {$L _0( \sigma ) := \lim _{m \to 0} \mathsf{ds} _m ( \sigma )$.}

\begin{lemma}[cf. \cite{Ca},\S11.4]\label{restW} Recall the $W_n$-representation
  $\{\mu,\nu\}$ from Lemma \ref{Wtypes}.
\begin{enumerate}
\item The restriction of $\{\mu, \nu \}$ to $W _n ^{\mathsf D}$ is irreducible unless $\mu = \nu$;
\item We have $\{ \mu, \nu \} \cong \{  \nu, \mu \}$ as $W _n ^{\mathsf D}$-modules;
\item We have {$\{ \mu, \nu \} \otimes \mathsf{sgn} \cong \{  {} ^{\mathtt t} \nu, {} ^{\mathtt t} \mu \}$};
\item The dimension of $\mathrm{Hom}_{W _n ^{\mathsf D}} ( \{ \mu,
    \nu \}, \{  \mu', \nu' \} )$ equals:
$$ \left\{\begin{matrix} 1,&\text{ if }(
        \mu, \nu ) \in \{( \mu',\nu' ),(\nu',\mu')\},\text{ but
        }\mu\neq\nu,\\ 
4,&\text{ if }(
        \mu, \nu ) \in \{( \mu',\nu' ),(\nu',\mu')\},\text{ and
        }\mu=\nu,\\ 0,&\text{ otherwise.}\end{matrix}\right.$$
\end{enumerate}
\end{lemma}

\begin{proposition}\label{irrD}
The restriction of $L_0( \sigma )$ from $\mathbb H _{n,0}$ to $\mathbb
H _{n} ^{\mathsf D}$ is irreducible.
\end{proposition}

\begin{proof}
In this proof, we freely identify $\mathbb H _n ^f$ with $\mathbb C [
  W _n ]$, and $\mathbb H _n ^{\mathsf D, f}$ with $\mathbb C [ W _n
  ^{\mathsf D} ]$.
 Applying the $\mathsf{ds} _m$-algorithm to $\sigma$ ($0
< m < \frac{1}{2}$), we deduce that the largest segment $I \in
\mathsf{ds} _m ( \sigma )$ (with respect to the cardinality) satisfies
$q ^m \in I$. Let $( \mu, \nu )$ be the bi-partition corresponding to
$G \mathcal O _{\mathsf{ds} _m ( \sigma )}$ by \cite{K2}
  Theorem 5.1. We have $L _{\sigma} := \{ {} ^{\mathtt t} \nu, {} ^{\mathtt t} \mu \} \subset L_0 ( \sigma )$ by the exotic Springer correspondence. If $\mathtt E _m ( I ) < 1$, then we deduce that $\mu _1 < \nu _1$ from the fact that
$I$ is not marked. If $\mathtt E _m ( I ) > 1$, then we have $\mu _1 >
\nu _1$. Thanks to Lemma \ref{restW} and \cite{K2} Theorem 10.7, we
conclude that the restriction of $L _{\sigma}$ from $W _n$ to $W _n
^{\mathsf D}$ is irreducible. {(Notice that the correspondences in
  \cite{K1} and \cite{K2} are equivalent under tensoring with
  $\mathsf{sgn}$, and the correspondences in \cite{K2} and Lemma
  \ref{restW} are equivalent, respectively.)}

By Lemma \ref{restW}, we have $[ L_0 ( \sigma ) : L _{\sigma} ] _{\mathbb C [ W _n
  ^{\mathsf D} ]} > 1$ only when $[ L_0 ( \sigma ) : L _{\sigma}' ] _{\mathbb C [ W _n ]} > 0$ with $L _{\sigma}' := \{ {} ^{\mathtt t} \mu, {} ^{\mathtt t} \nu \}$. In particular, there exists a {marked partition $\tau$ corresponding to $L_{\sigma}'$ and $\mathcal O _{\mathsf{ds} _m ( \sigma )} \subset G \mathcal O _{\mathsf{ds} _m ( \sigma )} \subset \overline{\mathcal O _{\tau}}$}.
  This happens only if $\mathtt E _m ( I ) < 1$ since we need $\mu _1 \le \nu _1$ by \cite{AH}. Moreover, we have $\mathsf{pr} ( G \mathcal O _{\mathsf{ds} _m ( \sigma )}) = \mathsf{pr} ( \mathcal O _{\tau})$ as $G$-orbits in $\mathbb V$, where $\mathsf{pr}$
is the projection $\mathbb V _n \to V ^{(2)} _n$.
We define $\Xi$ to be the set of
parameters $\chi$ which satisfy $\mathcal O _{\mathsf{ds} _m ( \sigma
  )} \subset \overline{\mathcal O _{\chi}}$, $\mathsf{pr} ( \mathcal O
_{\mathsf{ds} _m ( \sigma )} ) = \mathsf{pr} ( \mathcal O _{\chi} )$,
and $G \mathcal O _{\chi} \subset \overline{\mathcal O
  _{\tau}}$.  Recall that $L_\chi$ denotes the
  irreducible $\mathbb H_{n,m}$-module parameterized by $\chi$.
\begin{claim}
If $[L_\chi:L_\sigma'] _{\mathbb C [ W _n ]} \neq 0$ for some $\chi \in \Xi$, then $\chi$ must be
  maximal with respect to the closure ordering in $\Xi.$
\end{claim}
\begin{proof}
The variety $\mathsf{pr} ^{-1} ( \mathsf{pr} ( \mathcal O
_{\mathsf{ds} _m ( \sigma )}) ) \cap \mathbb V _n ^{( \mathsf c _m ^\sigma, \vec{q} )}$ is a vector bundle over $\mathcal O
_{\mathsf{ds} _m ( \sigma )}$. Hence, a closure relation in $\Xi$ just represents a
vector subbundle over $\mathsf{pr} ( \mathcal O _{\mathsf{ds} _m (
  \sigma )})$; this means that $\overline{\mathcal
  O_{\chi}}$ is smooth along $\mathcal O_{\chi'}$ for every
  $\chi,\chi'\in\Xi$ such that $\mathcal O_{\chi'}\subset \overline{\mathcal O_\chi}$. From this, we deduce $[
  M_{\mathsf{ds} _m ( \sigma )} : L _{\chi} ] = 1$
for $\chi \in \Xi$. By the analogous vector bundle structure for all of $\mathbb V _n$, we conclude $[ M_{\chi} : L _{\sigma}' ] _{\mathbb C[W _n]} =1$. Applying the Ginzburg theory, we conclude that $[ L _{\chi}
  : L _{\sigma}' ] _{\mathbb C[ W_n ]} = 1$ for a unique $\chi \in \Xi$. In view of
the above multiplicity estimates, we conclude that $\chi$ is maximal
in $\Xi$ with respect to the closure ordering.
\end{proof}
We return to the proof of Proposition \ref{irrD}.

In the case $\mathtt E _m ( I ) < 1$, we have {$G \mathcal O _{\mathsf{ds} _m ( \sigma )} \subset \overline{\mathcal O _{\chi'}} \subset \overline{\mathcal O _{\chi}}$ for a parameter $\chi'$ obtained from $\mathsf{ds} _m ( \sigma ) = ( \mathsf{c} _m ^\sigma, \tau )$ by adding an extra marking on $J \in \tau$ with $I = \underline{J}$.} This implies that $\mathsf{ds} _m ( \sigma )$ is not maximal in $\Xi$. Thus, we have necessarily $[ L_0( \sigma ): L _{\sigma}' ]_{\mathbb C [W _n]} = 0$. Therefore, we conclude $[ L_0 ( \sigma ) : L _{\sigma} ]_{\mathbb C [W ^{\mathsf D} _n]} = 1$.

We now show that $L_0 ( \sigma )$ is irreducible as an $\mathbb H _{n}
^{\mathsf D}$-module. {In order to deduce this by contradiction, we assume that} there exists
a proper $\mathbb H _{n} ^{\mathsf D}$-submodule $M \subset L_0 ( \sigma
)$. We have $\mathbb H _{n} ^{\mathsf D} L _{\sigma} = \mathbb H _{n}
^{\mathsf D} N_n L _{\sigma} = L _0 ( \sigma )$ since $\mathbb H _{n,0} =
\mathbb H _{n} ^{\mathsf D} + \mathbb H _{n} ^{\mathsf D}
N_n$. Because 
$[ L_0 ( \sigma ): L _{\sigma} ] _{\mathbb C [W ^{\mathsf D} _n]} = 1$, we
have $[ M, L _{\sigma} ] _{\mathbb C [W ^{\mathsf D} _n]} = 0$. Since $L_0
( \sigma )$ is irreducible as an $\mathbb H _{n,0}$-module, we have a surjection
$$\mathrm{Ind} ^{\mathbb H _{n,0}} _{\mathbb H _{n} ^{\mathsf D}} M = \mathbb H _{n,0} \otimes _{\mathbb H _{n} ^{\mathsf D}} M \longrightarrow \!\!\!\!\!\!\!\!\! \longrightarrow L_0 ( \sigma ).$$
As $W ^{\mathsf D} _n$-modules, $\mathrm{Ind} ^{\mathbb H _{n,0}}
_{\mathbb H _{n} ^{\mathsf D}} M=M\oplus N_n M$, and therefore $N_n M$
contains $L _{\sigma}$ as a $W ^{\mathsf D} _n$-module. However, $L
_{\sigma}$ must give rise to an irreducible $W
_n$-submodule of $L_0 ( \sigma )$. Therefore, we conclude $L _{\sigma}
\subset N_n M \cap M \subset M$. This is a contradiction and thus $L_0 ( \sigma )$ is irreducible as an $\mathbb H _{n}
^{\mathsf D}$-module.
\end{proof}

\section{Formal degrees}\label{sec:formal}

\subsection{Preliminaries}
In this section, we consider the Hecke algebra with three parameters $\mathbb H_n=\mathbb
H_n(q,u,v)$, and assume that $u$ and $v$ are
specialized to $u=q^{m_+}$ and $v=q^{m_-}$, where $q>1$ and $m_{\pm} \in \mathbb R$. We retain the notation from \S\ref{Hecke}. If $w\in \widetilde W_n$ has a reduced expression
$w=s_{i_1}\cdot\dots\cdot s_{i_k}$, $i_l\in\{0,1,\dots,n\}$, in terms of the
affine simple reflections, then define the elements of $\mathbb H_n$, $N_w=N_{i_1}\cdot\dots\cdot
N_{i_k}$, where $N_{i_l}$ are the generators of $\mathbb H_n$ from
\S\ref{Hecke}. This definition does not depend on the choice of
reduced expression.

The algebra $\mathbb H_n$ has a structure of normalized Hilbert algebra (in the sense of \cite{Di} A.54), with the $*$ operation given on generators by
\begin{equation*}
N_w^*=N_{w^{-1}},\quad w\in \widetilde W_n,
\end{equation*}
the trace functional $\tau$ given by 
\begin{equation*}
\tau(N_w)=0,\text{ if }w\neq 1,\text{ and }\tau(1)=1,
\end{equation*}
and the inner product $[~,~]$ given by 
\begin{equation*}
[x,y]=\tau(x^*y),\text{ for all }x,y\in\mathbb H_n.
\end{equation*}
Since all of the irreducible $\mathbb H_n$-modules are finite dimensional, the trace $\mathsf{tr}$ is well defined on every irreducible module, and there exists a positive Borel measure $\hat\mu$ on the tempered dual $\widehat{\mathfrak S}$ of $\mathbb H_n$ such that the abstract Plancherel formula holds:
\begin{equation}\label{plancherel}
[x,1]=\int_{\widehat{\mathfrak S}} \mathsf{tr}\pi(x)~d\hat\mu(\pi),\quad x\in\mathbb H.
\end{equation}
Moreover, an irreducible tempered representation $\pi$ has positive volume $\hat\mu(\pi)>0$ if and only if $\pi$ is a discrete series. In this case, we denote by $\fd(\pi)=\hat\mu(\pi)$ the formal degree of $\pi.$ The formal degree $\fd(\pi)$ is known up to a rational constant $C_\pi$ independent of $q$ (but depending on $\pi$). The purpose of this section is to calculate this constant.

To begin, we have the following known result. Recall that $R_n$ denotes
the set of roots of $T_n$ in $G_n=Sp(2n,\mathbb C)$, and let us denote
by $R_{n}^{\mathsf {sh}}$ and $R_n^{\mathsf {lo}}$ the short and the
long roots, respectively. 

\begin{theorem}[\cite{OS}, Theorem 4.6]\label{t:product} If $\pi$ is a discrete series of
  $\mathbb H_n$ with central character $s\in T_n$, there exists a rational
  constant $C_\pi$ independent of $q$ such that
\begin{equation}\label{e:product}
\fd(\pi)=\frac{{C_\pi~ q^{n^2-n} q^{n m_+}\prod_{\alpha\in R_n}' (\alpha(s)-1)}}{{\prod_{\alpha\in R_n^{\mathsf {sh}}}' (q\alpha(s)-1) \prod_{\alpha\in R_n^{\mathsf {lo}}}' (q^{\frac {m_++m_-}2}\alpha(s)^{1/2}-1)\prod_{\alpha\in R_n^{\mathsf {lo}}}' (q^{\frac {m_+-m_-}2}\alpha(s)^{1/2}+1) } },
\end{equation}
where $\prod'$ means that the product is taken only over the nonzero
factors. 
\end{theorem}

Notice
also that the known part of (\ref{e:product}) only depends on the
central character $W_n s,$ and not on $\pi.$ Our strategy in the determination of the constant in the formula (\ref{e:product}) is to compare this formula to an Euler-Poincar\'e formula which also gives the formal degree.

\medskip

The second idea we need, that of the Euler-Poicar\'e pairing, traces
back to Kottwitz and \cite{SS}. For the setting of the affine Hecke algebra, the reference is \cite{OSa}.

\begin{convention}\label{convdef}
If $S$ is a subset of
  $\widetilde\Pi_n$, let $\mathbb H_S^f$ and $W_S$ denote the finite
  Hecke subalgebra of $\mathbb H_n$ and the finite subgroup of $\widetilde W_n$ generated
  by the roots in $S$, respectively. By Tits' deformation theorem, we
  have an isomorphism of algebras $\mathbb H_S^f\cong \mathbb
  C[W_S]$. We use this identification in the following, and for
  example, for every $\sigma\in \widehat W_S$, we denote by
  $\gd(\sigma)$ the generic degree of the corresponding $\mathbb
  H_S^f$-module. There exists an
explicit formula for computing $\gd(\sigma)$ (see \cite{Ca}, page
447), which we recall later in (\ref{gd}).
\end{convention}

\begin{definition}[\cite{OSa}, equations (3.15),(3.19)] If $\pi$ is a finite dimensional $\mathbb H_n$-module, define the Euler-Poincar\'e element $f_\pi$ as follows:
\begin{equation}\label{e:EP}
f_\pi=\sum_{S\subsetneq \widetilde\Pi_n}
  (-1)^{n-|S|}\sum_{\gamma\in \mathrm{Irr}(\mathbb H_S^f)}\frac{[\pi :\gamma]_{\mathbb H_S^f}\gd(\gamma)}{\dim\gamma} e_\gamma,
\end{equation}
where  $e_\gamma\in \mathbb H_S^f$ is the primitive central idempotent in corresponding to $\gamma$. 
For $\pi,\pi'$ finite dimensional $\mathbb H_n$-modules, define the Euler-Poincar\'e pairing
\begin{equation}
\EP(\pi,\pi')=\sum_{i\ge 0}(-1)^i\dim\Ext^i_{\mathbb H_n}(\pi,\pi').
\end{equation}
\end{definition}

The remarkable property of $f_\pi$, established in this setting in
\cite{OSa} Proposition 3.6, is that one has
\begin{equation}
\mathsf{tr} \pi'(f_\pi)=\EP(\pi,\pi'),
\end{equation}
for all finite dimensional $\mathbb H_n$-modules $\pi,\pi'$.
Moreover, it is shown in \cite{OSa} Theorem 3.8, that 
\begin{equation}
\Ext^i_{\mathbb H_n}(\pi,\pi')=\mathbb C, \text{ if }\pi\cong \pi'
\text{ and }i=0,\quad  \Ext^i_{\mathbb H_n}(\pi,\pi')=0,\text{ otherwise,}
\end{equation}
whenever $\pi$ is an irreducible discrete series and $\pi'$ is an irreducible tempered
module. This allows one to use $f_\pi$ in
(\ref{plancherel}) to find the following formula for $\fd(\pi).$

\begin{theorem}[\cite{OSa}, Proposition 3.4]\label{t:sum} 
\begin{enumerate}
\item Let $\pi$ be a discrete
  series $\mathbb H_n$-module. The formal degree of $\pi$ is
\begin{equation}\label{e:sum}
\fd(\pi)=[f_\pi,1]=\sum_{S\subsetneq\widetilde\Pi_n}(-1)^{n-|S|}\sum_{\gamma\in \widehat W_S}\frac{[\pi:\gamma]_{\mathbb C[W_S]}\gd(\gamma)}{P_S},
\end{equation}
where  $P_S$ is the Poincar\'e polynomial for the Hecke algebra
$\mathbb H^f_S$. 
\item Assume that $\pi$ is irreducible tempered, but not a discrete series, or else, that it is a parabolically induced module from a discrete series on a proper Hecke subalgebra. Then, we have
\begin{equation}
[f_\pi,1]=0.
\end{equation}
\end{enumerate}
\end{theorem}

One can simplify formula (\ref{e:sum}), as in \cite{Re}. In the following, using Convention \ref{convdef}, we identify $\mathbb C[W_i\times W_{n-i}]$ with the corresponding finite Hecke subalgebra of $\mathbb H_{n}$.

\begin{corollary}\label{c:reeder}  Let $\pi$ be
    a finite dimensional $\mathbb H_n$-module. Then one has
\begin{equation}\label{e:reeder}
[f_\pi,1]=  \sum_{i=0}^n
  (-1)^{n-i}\sum_{\gamma_1\boxtimes\gamma_2\in \widehat {W_i\times
    W_{n-i}}}\frac{[\pi:\gamma_1\boxtimes\gamma_2]_{\mathbb C[ W_i\times W_{n-i} ]}\gd(\gamma_1)\gd(\gamma_2\otimes\sgn)}{P_{i}(q,q^{m_-})P_{n-i}(q,q^{m_+})},
\end{equation}
where $W_i\times W_{n-i}$ is the Coxeter group generated by the
reflections in the roots of $\widetilde \Pi_n$ except the $i$-th root
$($the roots are numbered $0,\dots,n)$, and $P_j$ denotes the Poincar\'e
polynomial for type $B_j$ and corresponding labels.
\end{corollary}

Every quantity in formula (\ref{e:reeder}) is computable, provided that
  we know the restrictions of $\pi$ to $\mathbb C[W_i\times W_{n-i}]\subset
  \mathbb H_{n}$, for every $0\le i\le n$.

\subsection{The constant in formal degrees}\label{s:mainformal}

\begin{definition}\label{d:generic} {A label $({m_+},{m_-})$, for the Hecke
  algebra $\mathbb H_n(q,q^{m_+},q^{m_-})$} is called generic if
  $|m_{+}\pm m_-|\notin\{0,1,2,\dots,2n-1\}$, and it is called critical otherwise.
\end{definition}

Let us recall briefly recall the part of the reduction to positive
real central character for $\mathbb H_n$ (\cite{L2}) that is
relevant to us. 
Assume $\pi$ is a discrete series of $\mathbb H_n$ with
central character $\mathsf c(\pi)\in T_n$, not necessarily
positive. Then there exists $k$, $1\le k\le n$, and two discrete
series $\pi_1$, $\pi_2$ of $\mathbb H_{k,m_1}$ and
$\mathbb H_{n-k,m_2}$, respectively, where $2|m_1|={m_+-m_-}$, 
$2|m_2|={m_++m_-}$, such that $\pi_1,\pi_2$ have positive real
central characters $\mathsf c(\pi_1)$, $\mathsf c(\pi_2)$, with
$\mathsf c(\pi_1)\times \mathsf c(\pi_2)\in T_k\times T_{n-k}=T_n$,
and
\begin{equation}\label{realred}
\begin{aligned}
&\mathsf c(\pi)=(-\mathsf c(\pi_1),\mathsf c(\pi_2)),\\
&\lim_{q\to 1}\pi=\Ind_{\widetilde W_k\times
  \widetilde W_{n-k}}^{\widetilde W_n}(\lim_{q\to
  1}\pi_1\boxtimes\lim_{q\to 1}\pi_2),\text{ and, in
  particular}\\
&\pi|_{W_n}=\Ind_{W_k\times W_{n-k}}^{W_n}(\pi_1|_{W_k}\boxtimes\pi_2|_{W_{n-k}});
\end{aligned}
\end{equation}
here $\pi_1|_{W_k}$ and $\pi_2|_{W_{n-k}}$ are understood as
restrictions in the Hecke algebra $\mathbb H_{k,m_1}$ and $\mathbb
H_{n-k,m_2}$, respectively, and $W_k, W_{n-k}$ in the induction are
viewed as the subgroups of $W_n$
generated by the reflections in the roots
$\{\epsilon_1-\epsilon_2,\dots,\epsilon_{k-1}-\epsilon_{k}, 2 \epsilon_k \}$
and
$\{\epsilon_{k+1}-\epsilon_{k+2},\dots,\epsilon_{n-1}-\epsilon_n,2\epsilon_n\}$, respectively. Here we remark that Algorithm \ref{algW} (and in particular Corollary \ref{c:link}) holds verbatim with respect to $m_1$ (for $\pi_1$) and $m_2$ (for $\pi_2$) independently.

We state the main result of this section.

\begin{theorem}\label{mainformal} Let $\pi$ be a discrete series for $\mathbb H_n$
  with central character $\mathsf c(\pi)$. If the parameters
  $(m_+,m_-)$ for $\mathbb H_n$ are generic, then the constant in formula
  (\ref{e:product}) for the formal degree $\fd(\pi)$ is (up to a
  sign) $C_\pi=1$. 
\end{theorem}

 The proof is
presented in the following subsections. The general case follows
immediately from Theorem \ref{mainformal}.

\begin{corollary}\label{c:mainformal} Let $\pi$ be a discrete series for $\mathbb H_n$
  with central character $\mathsf c(\pi)$. {Then in formula
  (\ref{e:product}), we have $C_\pi=\pm 1$.}
\end{corollary}

\begin{proof}
Assume that $\pi^0$ is a discrete series for $\mathbb H_{n}$ for
critical values $(m^0_{+},m^0_{-})$ of the labels, and having central
character $\mathsf c(\pi^0)$. 
Let $k,$ 
$m^0_{1}=\frac{m^0_{+}-m^0_{-}}2$, $m^0_2=\frac{m^0_{+}+m^0_{-}}2$, and discrete
series $\pi^0_1$, $\pi^0_2$ of $\mathbb H_{k,m^0_1}$ and
$\mathbb H_{k,m^0_2}$ be as in the discussion around
(\ref{realred}). Recall from \S\ref{delimits} that $\pi^0_1$ and $\pi^0_2$
belong to families of discrete series indexed by partitions $\sigma_1$
of $k$ and $\sigma_2$ of $n-k$, respectively. To emphasize this
dependence, we write, as in \S\ref{delimits},
$\pi^0_i=\ds_{m^0_i}(\sigma_i)=(\mathsf
c_{m^0_i}^{\sigma_i},ds_{m^0_i}(\sigma_i))$, $i=1,2$, where $\mathsf
c_{m^0_i}^{\sigma_i}$ is the central character, and
$ds_{m^0_i}(\sigma_i)$ is the marked partition adapted to $\mathsf
c_{m^0_i}^{\sigma_i}$ that parameterizes the discrete series.

Set $m^t_i=m^0_i+t$ for $t\in (-\frac 12,\frac 12)$ and $i=1,2$, and
consider the corresponding discrete series $\pi^t_i=\ds_{m^t_i}(\sigma_i)=(\mathsf
c_{m^t_i}^{\sigma_i},ds_{m^t_i}(\sigma_i)).$ Then, by Corollary
\ref{irred}, we have 
\begin{equation}\label{reallim}
\lim_{t\to
  0}\pi^t_i=\pi^0_i, \quad i=1,2. 
\end{equation}
For every $t\in (-\frac 12,\frac 12)$, set $m^t_+=m^t_1+m^t_2$ and
$m^t_-=m^t_2-m^t_1=m_{0,-}$, and let $\pi^t$ be the discrete series
module of
$\mathbb H_n(q,q^{m^t_+},q^{m^t_-})$ corresponding in the reduction to
positive real central character to the pair
$(\pi^t_1,\pi^t_2)$. Notice that the labels $(m^t_+,m^t_-)$ are
generic, in the sense of Definition \ref{d:generic}, for all
$t\in(-\frac 12,\frac 12)\setminus\{0\}.$ Moreover (\ref{reallim})
implies that  $\lim_{t\to 0}\pi^t=\pi^0$. This, in conjunction with
Theorem \ref{t:sum} 1), gives
\begin{equation}
\lim_{t\to
  0}\fd(\pi^t)=\fd(\pi^0).
\end{equation}
Since we assumed that $\pi$ is a discrete
series, this limit must be nonzero. Therefore, one only needs to analyze the factors in
(\ref{e:product}) for $\pi^t$ that vanish at $t\to 0$. More precisely,
with the notation
\begin{align*}
&R_n(j):=\{\alpha\in R_n:~\alpha(\mathsf c(\pi^0))=q^j\},\quad R_n^{\mathsf {lo}}(j)_\pm:=\{\alpha\in R_n^{\mathsf{lo}}:~
\alpha(\mathsf c(\pi^0))^{1/2}=\pm q^j\},
\end{align*}
we have:
\begin{equation}\label{Clim}
\begin{aligned}
C_{\pi^0}=\lim_{t\to 0}&\frac{C_{\pi^t}\prod'_{\alpha\in
    R_n(0)}(\alpha(\mathsf c(\pi^t))-1)}{\prod_{\alpha\in
    R_n(-1)\cap R_n^{\mathsf{sh}}}'(q\alpha(\mathsf c(\pi^t))-1)\prod'_{\alpha\in
    R_n^{\mathsf{lo}}(-m_2^0)_+}(q^{m_2^t}\alpha(\mathsf
  c(\pi^t))^{1/2}-1)}\\
&\cdot\frac 1{\prod'_{\alpha\in
    R_n^{\mathsf{lo}}(-m_1^0)_-}(q^{m_1^t}\alpha(\mathsf
  c(\pi^t))^{1/2}+1)}.
\end{aligned}
\end{equation}

Firstly, notice that for all roots of the form
$\alpha=\pm\epsilon_i\pm\epsilon_j$ such that $1\le i\le k$, $k+1\le j\le n$,
we have $\alpha(\mathsf c(\pi^t))<0$ and therefore, these roots do not
appear in (\ref{Clim}).
Secondly, we remark that for all the roots $\alpha$ (both short and long) that appear in (\ref{Clim}),
 the
corresponding factors must
be of the form $\pm(q^{\pm 2t}-1),$ where $t\to
0$. This implies that we have $C_{\pi^0}=\lim_{t\to 0} C_{\pi^t},$ and
this proves the claim.
\end{proof}

Before presenting the proof of Theorem \ref{mainformal}, we explain how the result in Corollary \ref{c:mainformal} relates to
the expected form of the formal degree in the case of affine Hecke
algebras of $\mathbb H_{n,m}$, $\mathbb H_{n,m}'$, $\mathbb
H_{n}^{\mathsf D}$ of types $\mathsf C_n, \mathsf B_n,
\mathsf D_n$, respectively. To emphasize the type of the root system,
  let
$R^{\mathsf C}_n,R^{\mathsf B}_n,R^{\mathsf D}_n,$ denote the roots in
these cases. We have $R^{\mathsf C}_n=R_n$, $R^{\mathsf
  B}_n=\check R_n,$ and $R^{\mathsf D}_n=R_n^{\mathsf {sh}}$.

\smallskip

(1) For $\mathbb H_{n,m}$, we specialize $m_+=m,$ $m_-=m$. Assuming
$\pi^{\mathsf C}$ is a discrete series with central character $s$, we
find:
\begin{equation}
\begin{aligned}
\fd(\pi^{\mathsf C})&=\frac{q^{n^2-n}q^{nm} \prod_{\alpha\in
    R_n}' (\alpha(s)-1)}{\prod_{\alpha\in R_n^{\mathsf {sh}}}'
  (q\alpha(s)-1) \prod_{\alpha\in R_n^{\mathsf {lo}}}'
  (q^{m}\alpha(s)^{1/2}-1)\prod_{\alpha\in R_n^{\mathsf {lo}}}'
  (\alpha(s)^{1/2}+1)  }\\
&={ C^{\mathsf C}_{\pi}} \ \frac{q^{n^2-n}q^{nm}\prod_{\check\alpha\in
    \check R^{\mathsf C}_n}' (\check\alpha(s)-1)}{\prod_{\check\alpha\in \check
    R^{\mathsf C}_n}'
  (q(\alpha)\check\alpha(s)-1)}, \quad
  q(\alpha)=\left\{\begin{matrix}&q,\quad &\alpha\text{
        short}\\&q^m,\quad &\alpha\text{ long}\end{matrix}\right.,
\end{aligned}
\end{equation}
where
\begin{equation}\label{typeC}
C^{\mathsf C}_\pi=\frac 1{2^{e^{\mathsf C}}},\quad e^{\mathsf C}=2\#\{i: 1\le i\le n,~ \epsilon_i(s)=1\}.
\end{equation}

(2) For $\mathbb H'_{n,m}$, we specialize $m_+=2m,$ $m_-=0$. Assuming
$\pi^{\mathsf B}$ is a discrete series with central character $s$, we
find:
\begin{equation}
\begin{aligned}
\fd(\pi^{\mathsf B})&=\frac{q^{n^2-n}q^{2nm} \prod_{\alpha\in
    R_n}' (\alpha(s)-1)}{\prod_{\alpha\in R_n^{\mathsf {sh}}}'
  (q\alpha(s)-1) \prod_{\alpha\in R_n^{\mathsf {lo}}}'
  (q^{m}\alpha(s)^{1/2}-1)\prod_{\alpha\in R_n^{\mathsf {lo}}}'
  (q^m\alpha(s)^{1/2}+1)  }\\
&={ C^{\mathsf B}_{\pi}}\ \frac{q^{n^2-n}q^{2nm}\prod_{\check\alpha\in
    \check R^{\mathsf C}_n}' (\check\alpha(s)-1)}{\prod_{\check\alpha\in \check
    R^{\mathsf C}_n}'
  (q(\alpha)\check\alpha(s)-1)}, \quad
  q(\alpha)=\left\{\begin{matrix}&q,\quad &\alpha\text{
        short}\\&q^{2m},\quad &\alpha\text{ long}\end{matrix}\right.,
\end{aligned}
\end{equation}
where
\begin{equation}\label{typeB}
C^{\mathsf B}_\pi=\frac 1{2^{e^{\mathsf B}}},\quad e^{\mathsf B}=\#\{i: 1\le i\le n,~
\epsilon_i(s)=\pm q^{m}\}+\#\{i: 1\le i\le n,~ \epsilon_i(s)=\pm q^{-m}\}.
\end{equation}

(3) For $\mathbb H_{n}^{\mathsf D}$, we specialize $m_+=0,$ $m_-=0$. Assuming
$\pi^{\mathsf D}$ is a discrete series with central character $s$, we
find:
\begin{equation}
\begin{aligned}
\fd(\pi^{\mathsf D})&=\frac{q^{n^2-n} \prod_{\alpha\in
    R_n}' (\alpha(s)-1)}{\prod_{\alpha\in R_n^{\mathsf {sh}}}'
  (q\alpha(s)-1) \prod_{\alpha\in R_n^{\mathsf {lo}}}'
  (\alpha(s)^{1/2}-1)\prod_{\alpha\in R_n^{\mathsf {lo}}}'
  (\alpha(s)^{1/2}+1)  }\\
&={ C^{\mathsf D}_{\pi}} \ \frac{q^{n^2-n}\prod_{\check\alpha\in
    \check R^{\mathsf D}_n}' (\check\alpha(s)-1)}{\prod_{\check\alpha\in \check
    R^{\mathsf D}_n}'
  (q\check\alpha(s)-1)},
\end{aligned}
\end{equation}
where
\begin{equation}\label{typeD}
C^{\mathsf D}_\pi=\frac 1{2^{e^{\mathsf D}}},\quad e^{\mathsf D}=2\# \{i: 1\le i\le n,~
\epsilon_i(s)=\pm 1\}.
\end{equation}
In types $\mathsf B$ and $\mathsf D$, we needed to account for central extensions.

\subsection{The proof in the generic case for positive central character}

\begin{convention}We restrict, as we may, to the specialization $\mathbb H_{n,m}$. If $\sigma$ is a partition of $n$, recall from \S\ref{delimits} that there is a
real central character ${\mathsf c}_m^\sigma$ attached to it. The Hecke algebra $\mathbb
H_{n,m}$  admits a family of discrete series $\ds_m(\sigma)$ with
central character ${\mathsf c}_m^\sigma$.
\end{convention}

The starting point is the case $m\to\infty$ (so $m>n-1$), when the $W_n$-character  is easy to understand.

We recall the formula for the generic degree (see \cite{Ca}, section
13.5) of the
module of the finite Hecke algebra $\mathbb H^f_n(u,v)$ of
type $\mathsf C_n$ with parameters $u$ on the short roots and $v$ on the long
roots corresponding to 
$\gamma=\{(a_1,\dots,a_{k+1})(b_1,\dots,b_k)\}\in \widehat W_n$,
in the bipartition notation. Here
each partition is written in nondecreasing order, and assume
that at least one of $a_1$ or $b_1$ are nonzero. Then, form the symbol
$\left(\begin{matrix}\lambda_1 & &\lambda_2&\dots
  &\lambda_k&&\lambda_{k+1}\\ &\mu_1
  &&\mu_2&\dots&\mu_k\end{matrix}\right),$ where $\lambda_i=a_i+(i-1)$
  and $\mu_j=b_j+(j-1).$ The generic degree $\gd(\gamma)$ with parameters $u$
  and $v$ is
\begin{equation}\label{gd}
\displaystyle{\frac{u^{m+{m\choose 2}} v^{\sum_j b_j} P_n(u,v)
    (u-1)^n\prod_{i'<i}(u^{\lambda_i}-u^{\lambda_{i'}})\prod_{j'<j}(u^{\mu_j}-u^{\mu_{j'}})\prod_{i,j}(u^{\lambda_i-1}v+u^{\mu_j})}{u^{{{2m-1}\choose
    2}+{{2m-3}\choose
    2}+\dots}
    \left(\prod_i\prod_{l=1}^{\lambda_i}(u^l-1)(u^{l-1}v+1)\right)\left(\prod_j\prod_{l=1}^{\mu_j}(u^l-1)(u^{l+1}+v)\right) (u+v)^k
   }.
          }
\end{equation}

\begin{proposition}\label{p:dsinfty} Assume that $m>n-1.$ In
  $\mathbb H_{n,m}$, we have $C_{\ds_m(\sigma)}= \pm 1$, for every
  partition $\sigma$ of $n$.
\end{proposition}

\begin{proof}
From \cite{CK} \S4.7 we have
\begin{equation}
\ds_m(\sigma)|_{W_n}=\{\emptyset,{}^{\mathtt t}\sigma\}.
\end{equation}
We
identify the lowest power of $q$ in the Euler expansion
(\ref{e:reeder}). This expansion becomes:
\begin{equation}\label{e:alt}
(-1)^n[f_\pi,1]=\frac{\gd\{\sigma,\emptyset\}}{P_n}+\sum_{i=1}^n(-1)^{i}\sum_{\gamma_1\boxtimes\gamma_2\in \widehat {S_i\times
    S_{n-i}}}\frac{[\sigma:\gamma_1\boxtimes\gamma_2]_{\widehat
    {S_i\times
    S_{n-i}}}\gd\{\gamma_2,\emptyset\}\gd\{\emptyset,{}^{\mathsf t}\gamma_1\}}{P_{i}P_{n-i}}.
\end{equation}

We are interested, in the case when 
$m\to\infty,$ to determine the lowest degree of $q$ which may appear
in the terms from (\ref{e:alt}). We use the generic degree formula
(\ref{gd}) for $u=q$, $v=q^m$. The observation, using (\ref{gd}), is that for
every $i\ge 1$, each factor of the form $\gd\{\emptyset,{}^{\mathsf
  t}\gamma_1\}$ in (\ref{e:alt}) contains a factor $v^i$ and this dominates all factors
in $u$. Therefore the lowest degree
of $q$
in every term for $i\ge 1$ is a linear nonconstant function in
$m$. Moreover, in the first term in (\ref{e:alt}) (corresponding to
$i=0$), since the bipartition is $\{\sigma,\emptyset\}$, there is no
factor of $v$ present, and the lowest degree factor for $m\gg 0$ is a power of
$u=q$ independent of $m$. This implies that the lowest power of $q$ in the right
hand side of 
(\ref{e:alt}) appears with coefficient one.
On the other
hand, in the product formula for $\fd(\pi)$ (see (\ref{e:product})),
clearly the coefficient of the lowest power in a $q$-expansion is
$\pm C_\pi.$
\end{proof}

\begin{remark}Let $\sigma=(0<a_1\le a_2\le \dots \le a_{l+1})$ denote a partition which we identify with the corresponding irreducible $\mathfrak S_n$-module. We define its lowest harmonic degree as $\text{\rm lhd}(\sigma)=\sum_{j=1}^l (l+1-j) a_j$. It is an elementary combinatorial calculation to see
  that the lowest powers of $q$ in both (\ref{e:alt}) and
  (\ref{e:product}) for $\ds_m(\sigma)$ are $q^{\text{\rm lhd}(\sigma)}$.
\end{remark}

\begin{theorem}\label{t:gen}
Assume that $m$ is generic. Let $\sigma$ be a partition of $n$. Then,
in $\mathbb H_{n,m}$, we have 
$C_{\ds_m(\sigma)}=\pm 1$.

\end{theorem}

\begin{proof}
We analyze the behavior of $\fd(\ds_m(\sigma))$ as $m$
crosses a critical value $m_0\in \frac 12\mathbb Z.$ Let $h$ denote
the number of balanced hooks at $m_0$ in $\sigma.$ If $\pi$ is any
$\mathbb H_{n,m}$-module, recall that $\Theta_W(\pi)$ denotes the
$W_n$-character of $\pi.$

 Let $\ds^{\rightarrow}_m(\sigma),
  \ds^{\leftarrow}_m(\sigma)$ be the discrete series modules parameterized by
  $\sigma$ for $m_0-\frac 12<m<m_0$ and $m_0<m<m_0+\frac 12$,
  respectively. Then there exists a tempered delimit
  $\pi_m^{\rightarrow}\in \mathcal D_{m _0} ( \sigma )$ such that $$\Theta_W(\ds^{\rightarrow}_m(\sigma))=\Theta_W(\pi_m^{\rightarrow}).$$
It is clear that also $[f_{\ds^{\rightarrow}_m(\sigma)},1]=[f_{\pi_m^{\rightarrow}},1],$ since $[f_\pi,1]$
depends only on the $W$-character (for real positive central
character). Now using Corollary \ref{c:link} and Theorem \ref{t:sum} (2), we find that 
$[f_{\pi_m^{\rightarrow}},1]=\pm [f_{\ds^{\leftarrow}_m(\sigma)},1]$, and therefore
$[f_{\ds^{\rightarrow}_m(\sigma)},1]=\pm [f_{\ds^{\leftarrow}_m(\sigma)},1]$, which implies the claim of
Theorem \ref{t:gen}.\end{proof}

\subsection{The proof in the generic case for nonpositive central character}

\begin{convention}\label{convnon}As in the discussion around (\ref{realred}), let $\pi$, $\pi_1$, $\pi_2$ be
discrete series of $\mathbb H_n$, $\mathbb H_{n,m_1}$, $\mathbb
H_{n,m_2}$, respectively. We retain the notation from (\ref{realred}). In addition,
we regard $\pi_1$ and $\pi_2$ as part of the families
$\ds_{m_1}(\sigma_1)$ and $\ds_{m_2}(\sigma_2)$ for partitions
$\sigma_1$ of $k$ and $\sigma_2$ of $n-k$, respectively. Consequently,
we denote the discrete series $\pi$ of $\mathbb H_n$ by $\ds_{(m_+,m_-)}(\sigma_1,\sigma_2)$.
\end{convention}

The proof is analogous to the case of positive real central character
for $\mathbb H_{n,m}$. The analogous asymptotic region that we need is:
\begin{equation}\label{asy}
m_-\gg m_+\gg 0:\quad m_1=\frac{m_+-m_-}2\to -\infty,\quad m_2=\frac{m_++m_-}2\to +\infty.
\end{equation}
Again by \cite{CK} \S4.7, we have that
\begin{equation}\label{eq:4.21}
\ds_{m_1}(\sigma_1)|_{W_{k}}=\{\sigma_1,\emptyset\},\quad
\ds_{m_2}(\sigma_2)|_{W_{n-k}}=\{\emptyset,{}^{\mathtt t}\sigma_2\}.
\end{equation}
From this and (\ref{realred}), we see that
\begin{equation}
\ds_{(m_+,m_-)}(\sigma_1,\sigma_2)|_{W_n}=\{\sigma_1,{}^{\mathtt t}\sigma_2\}.
\end{equation}
In order to apply (\ref{e:reeder}) we need to analyze the
restrictions $\ds_{(m_+,m_-)}(\sigma_1,\sigma_2)|_{W_i\times W_{n-i}}.$
The following lemma
is sufficient for our purposes.

\begin{lemma}\label{affinerestr} Assume that $m_-\gg m_+\gg 0$. If $\Hom_{W_i\times
    W_{n-i}}(\gamma\boxtimes \delta,\ds_{(m_+,m_-)}(\sigma_1,\sigma_2))\neq
  0$ and $i\ge 1$, where $\gamma\in \widehat W_i$ and $\delta\in\widehat
  W_{n-i}$, then $\Hom_{W_1}(\sgn,\gamma)\neq 0,$ where $W_1\subset
  W_i$ denotes the reflection group generated by $s_0:=N_0|_{q=1}.$
\end{lemma}

\begin{proof}An algebraic family of modules of a finite group is
  rigid, and hence we have
$$\ds_{(m_+,m_-)}(\sigma_1,\sigma_2)\cong \lim_{q\to
  1}\ds_{(m_+,m_-)}(\sigma_1,\sigma_2),\text{ as }W_i\times W_{n-i}\text{-modules},$$
for every $0\le i\le n.$ 

Therefore, we replace $\ds_{m_1}(\sigma_1),\ds_{m_2}(\sigma_2),
\ds_{(m_+,m_-)}(\sigma_1,\sigma_2)$ with their limits $q\to 1$ during
this proof. We have  presentations
$$\mathbb C[\widetilde W_k]\cong \mathbb C[W_k]\otimes \mathbb
C[\epsilon_1^{\pm 1},\dots,\epsilon_k^{\pm 1}],\text{ and } \mathbb C[\widetilde W_{n-k}]\cong \mathbb C[W_{n-k}]\otimes \mathbb
C[\epsilon_{k+1}^{\pm 1},\dots,\epsilon_n^{\pm 1}],$$
inside 
$$\mathbb C[\widetilde W_n]\cong \mathbb C[W_n\ltimes
X^*(T_n)]=\mathbb C[W_n]\otimes \mathbb
C[\epsilon_1^{\pm 1},\dots,\epsilon_n^{\pm 1}].$$
By examining the central characters, we deduce that each
$\epsilon_1,\dots,\epsilon_k$ acts on $\ds_{m_1}(\sigma_1)\subset
\ds_{(m_+,m_-)}(\sigma_1,\sigma_2)$ by the uniform eigenvalue $-1$,
while each
$\epsilon_{k+1},\dots,\epsilon_n$ acts on $\ds_{m_2}(\sigma_2)\subset
\ds_{(m_+,m_-)}(\sigma_1,\sigma_2)$ by the uniform eigenvalue
$1$. Moreover, by (\ref{eq:4.21}), a long reflection of $W_i$ acts on
$\ds_{m_1}(\sigma_1)$ by the identity, while a long reflection of
$W_{n-i}$ acts on $\ds_{m_2}(\sigma_2)$ by the negative of the
identity.

We have $s_0=N_0|_{q=1}=\epsilon^{-1}_1\cdot s_\theta$, where $s_\theta$
is the long reflection of $W_n$ corresponding to
$\theta=2\epsilon_1$. In particular, we have $s_0\in \widetilde W_n$
and its $W_n$-conjugate act on 
$$\ds_{m_1}(\sigma_1)\boxtimes \ds_{m_2}(\sigma_2)\subset
\ds_{(m_+,m_-)}(\sigma_1,\sigma_2)$$
with uniform eigenvalue $-1$. Since
$$\mathbb C[W_n](\ds_{m_1}(\sigma_1)\boxtimes
\ds_{m_2}(\sigma_2))=\ds_{(m_+,m_-)}(\sigma_1,\sigma_2),$$
the same is true for $\ds_{(m_+,m_-)}(\sigma_1,\sigma_2).$

The semisimplicity of the complex representations of finite
groups implies then that $s_0$ acts on $\ds_{(m_+,m_-)}(\sigma_1,\sigma_2)$
by the negative of the identity. 

Therefore, we conclude that the claim holds via the intermediate
restriction $\langle s_0\rangle\subset W_i\times W_{n-i}\subset
\widetilde W_n.$
\end{proof}

\begin{proposition}\label{p:asynon} Assume that $m_-\gg m_+\gg 0.$ Fix $1\le k\le n$
    and partitions $\sigma_1,\sigma_2$ of $k$ and $n-k$,
    respectively. Let $\pi:=\ds_{(m_+,m_-)}(\sigma_1,\sigma_2)$ be as in
    Convention \ref{convnon}. Then we have $C_\pi=\pm 1.$
\end{proposition}
\begin{proof} The proof is analogous to the proof of
  Proposition \ref{p:dsinfty}. In this case, the analogue of (\ref{e:alt})  is the formula:
\begin{equation}\label{e:alt2}
(-1)^n[f_\pi,1]=\frac{\gd\{\sigma_2,{}^{\mathtt t}\sigma_1\}}{P_n(q,q^{m_+})}+\sum_{i=1}^n(-1)^i\sum_{\gamma\boxtimes\delta\in \widehat {W_i\times W_{n-i}}}\frac{[\pi:\gamma\boxtimes\delta]\gd(\gamma)\gd(\delta\otimes\sgn)}{P_{i}(q,q^{m_-})P_{n-i}(q,q^{m_+})}.
\end{equation}
In this formula, $\gd(\gamma)$ is computed in $\mathbb
H^f_i(q,q^{m_-})$, while $\gd\{\sigma_2,{}^{\mathtt t}\sigma_1\}$ and
$\gd(\delta)$ in $\mathbb H^f_n(q,q^{m_+})$ and $\mathbb
H^f_{n-i}(q,q^{m_+})$, respectively. In light of Lemma
\ref{affinerestr}, every $\gamma\in \widehat W_i$, $i\ge 1$ that appears
in (\ref{e:alt2}) contains the sign representation in its restriction
to $W_1$. In other words, when written in the bipartition notation,
$\gamma=\{\gamma_1,\gamma_2\}$, we have $\gamma_2\neq\emptyset$. The
formula for generic degree (\ref{gd}) implies then that $\gd(\gamma)$
contains the factor $v^{|\gamma_2|},$ where $v=q^{m_-}$, and $|\gamma_2|$ is
the size of $\gamma_2$.  Since $m_-\gg m_+\gg 0$, the smallest power of $q$ in the
$q$-expansion is in the first term of (\ref{e:alt2}), this
being the only term of the sum that does not have as a factor a power
of $q^{m_-}$. In particular, the coefficient of the lowest power of $q$ in $[f_\pi,1]$
in the region (\ref{asy}) is $\pm 1.$
\end{proof}

Now we can prove Theorem \ref{mainformal} in the nonpositive case
as well.

\begin{proof}[Proof of Theorem \ref{mainformal}] 
Fix $k$ such that $1\le k\le n$, and fix $\sigma_1,\sigma_2$
  partitions of $k$ and $n-k$, respectively. Recall the family  of
  discrete series $\ds_{(m_+,m_-)}(\sigma_1,\sigma_2)$ as in
  Convention \ref{convnon}. Proposition \ref{p:asynon} verified the
  claim of Theorem \ref{mainformal} in the asymptotic region $m_-\gg
  m_+\gg 0.$

To transfer the result from the asymptotic region to every generic
region, one can proceed as in the proof of Theorem
\ref{t:gen}. 

 Fix a pair $(m_1^{-\infty},m_2^\infty)$ in the
asymptotic region (\ref{asy}), and construct a line of parameters
$t\mapsto (m_1^t,m_2^t)$, $t\ge 0,$ where $m_1^t=m_1^{-\infty}+t$,
$m_2^t=m_2^{\infty}-t,$ a line of central
characters $\mathsf c^t(\sigma_1,\sigma_2):=(-\mathsf
c_{m^t_1}^{\sigma_1},\mathsf c_{m^t_2}^{\sigma_2})$, and the
  corresponding lines of discrete series
  $\pi_1^t=\ds_{m_1^t}(\sigma_1)$, $\pi_2^t=\ds_{m_2^t}(\sigma_2)$,
  $\pi^t=\ds_{(m^t_+,m^t_-)}(\sigma_1,\sigma_2)$ of $\mathbb H_{k,m_1^t}$,
  $\mathbb H_{n-k,m_2^t}$, and $\mathbb H_n$, respectively. Recall
  that $m_+^t=m_2^t+m_1^t$ (and hence $m_+^t=m_+^\infty$ is
  independent of $t$), and $m_-^t=m_2^t-m_1^t.$ (Notice that these
  lines of parameters cover all possible values of the labels
  $(m_+,m_-)$ as $(m_1^{-\infty},m_2^\infty)$ varies in the asymptotic region.)

Let $t=t_0>0$ be a value for which the label $(m_+^{t_0},m_-^{t_0})$ is
critical, and let $U_\epsilon(t_0):=(t_0-\epsilon,t_0+\epsilon)$ be an interval such
that the label $(m_+^{t},m_-^{t})$ is generic for $t\in
U_\epsilon(t_0)\setminus\{0\}$. By induction, we may assume that we
know $C_{\pi^t}$ for all $t\in U_\epsilon(t_0)$ with $t<t_0$. As in
the proof of Theorem \ref{t:gen}, we know that $\pi_i^t$, $\pi_i^{t'}$,
$t_0-\epsilon<t<t_0<t'<t_0+\epsilon$ are linked, $i=1,2$. The
reduction to positive real central character implies that also $\pi^t$
and $\pi^{t'}$ are linked. But then again Theorem \ref{t:sum} 2)
gives $C_{\pi^t}=\pm C_{\pi^{t'}}.$
\end{proof}

\begin{remark}
In order to apply these proofs to obtain the type $D$ formulas in
\S\ref{s:mainformal}, in light of Proposition \ref{irrD}, it is
sufficient to notice that if we
  have a 
$W_n$-type $\{\mu,\mu\}$ which splits 
  $\{\mu,\mu\}=\{\mu,\mu\}_+\oplus\{\mu,\mu\}_-$ as
  $W_n^{\mathsf D}$-representations, then $\frac 12\lim_{m \to 0}\gd
  \{\mu,\mu\}=\gd^{\mathsf D}\{\mu,\mu\}_+=\gd^{\mathsf D}\{\mu,\mu\}_-$
  (\cite{Ca}, \S 13.5). 
\end{remark}

}

\end{document}